\documentclass[11pt,reqno]{amsart}
\usepackage{amssymb,mathrsfs,graphicx,enumerate}
\usepackage{amsmath,amsfonts,amssymb,amscd,amsthm,bbm}
\usepackage{graphicx,colortbl,here}
\usepackage{xcolor}
\numberwithin{figure}{section}
\usepackage[utf8]{inputenc}
\usepackage[T1]{fontenc}
\usepackage{bm}
\usepackage{etoolbox}

\def\d {{\partial}}

\def\indc {{\bf 1}}

\def \rpsi_i {|\psi_i \rangle}
\def \lpsi_i {\langle \psi_i|}
\def \lrpsi_i{\langle \psi_i | \psi_i \rangle}
\def \rpsi_k {|\psi_k \rangle}
\def \lpsi_k {\langle \psi_k|}
\def \lrpsi_k{\langle \psi_k | \psi_k \rangle}

\newcommand{\Span}{\operatorname{span}}

\newcommand{\Id}{\operatorname{Id}}

\newcommand{\bbr}{\mathbb R}

\newcommand{\bbz}{\mathbb Z}

\def\d{\mathrm{d}}

\newcommand{\dm}{d}

\newcommand{\ba}{\begin{aligned}}
\newcommand{\ea}{\end{aligned}}

\newcommand{\be}{\begin{equation}}
\newcommand{\ee}{\end{equation}}

\renewcommand{\d}{  {\textup{d}} }

\theoremstyle{plain}
\newtheorem{theorem}{Theorem}[section]
\newtheorem{lemma}[theorem]{Lemma}
\newtheorem{corollary}[theorem]{Corollary}
\newtheorem{proposition}[theorem]{Proposition}

\theoremstyle{definition}
\newtheorem{definition}[theorem]{Definition}

\theoremstyle{remark}
\newtheorem{remark}[theorem]{Remark}

\numberwithin{equation}{section}

\begin{document}

\title[Fokker--Planck Equation with Simplex-Volume Interactions]
{Equilibrium selection and sharp convergence rates for a nonlinear Fokker--Planck equation with simplex-volume interactions}

\author[D. Kim]{Dohyun Kim}
\address[D. Kim]{\newline Department of Mathematics Education and Institute of Pure and Applied Mathematics, \newline Sungkyunkwan University, Seoul 03063, Republic of Korea, \newline School of Computational Sciences, \newline Korea Institute for Advanced Study, Seoul 02455, Republic of Korea}
\email{dohyunkim@skku.edu}

\author[H. Park]{Hansol Park}
\address[H. Park]{\newline Department of Mathematics, \newline National Tsing Hua University, Hsinchu 30013, Taiwan}
\email{hansolpark@math.nthu.edu.tw}

\author[W. Shim]{Woojoo Shim}
\address[W. Shim]{\newline Department of Mathematics Education, \newline Kyungpook National University, Daegu 41566, Republic of Korea}
\email{wjshim@knu.ac.kr}

\thanks{\textbf{Acknowledgment.}
The work of D. Kim was supported by the National Research Foundation of Korea (NRF) grant funded by the Korean government (MSIT) (RS-2024-00454452) and by the Visiting Professorship at Korea Institute for Advanced Study, and the work of H. Park was supported by the National Science and Technology Council (NSTC), Taiwan (Grant No. NSTC 115-2115-M-007-001-MY3). The work of W. Shim was supported by the National Research Foundation of Korea (NRF) grant funded by the Korean government (MSIT) (RS-2026-25605588)}

\begin{abstract}
We study a nonlinear Fokker--Planck equation with diffusion and nonlocal interactions generated by squared simplex volumes. The associated free energy is not displacement semiconvex when the simplex dimension $n$ is at least two. For $1\leq n<d$, the stationary states are isotropic Gaussian measures, and every solution converges to the equilibrium determined by its conserved mean. At $n=d$, the free energy fixes only the covariance determinant, allowing a noncompact continuum of anisotropic Gaussian equilibria when $d\geq2$. The limit is
then selected by the initial covariance, whose eigenvectors and
eigenvalue gaps are preserved by the flow. We establish global
well-posedness for arbitrary initial probability measures with finite second moment, identify stationary states with the global minimizers of the free energy, and prove convergence in quadratic Wasserstein distance with explicit exponential rates. For compactly supported initial data with nonstationary covariance, the sharp rate is the smaller of the covariance relaxation rate and three times the slowest transport
rate, provided that the third moments do not vanish along the
directions of slowest transport. The skewness of the initial datum, rather than its covariance, may therefore dictate the long-time behavior.
\end{abstract}

\keywords{nonlinear Fokker--Planck equation, McKean--Vlasov equation,
simplex-volume interaction, covariance dynamics, Gaussian equilibrium,
equilibrium selection, convergence rate}

\makeatletter
\@namedef{subjclassname@2020}{%
  \textup{2020} Mathematics Subject Classification}
\makeatother

\subjclass[2020]{35Q84, 49Q22, 35B40, 35A01, 35R06}


\date{\today}

\maketitle


\section{Introduction}
\setcounter{equation}{0}

Nonlocal Fokker--Planck equations describe the evolution of probability measures under the combined effects of diffusion and measure-dependent transport. In the presence of attractive interactions, diffusion tends to spread mass, whereas transport may promote concentration or spatial organization. The competition between these effects is often reflected in the structure of stationary states and the long-time behavior of solutions.

A widely studied class involves pairwise interaction energies of the form
$\frac12\iint W(x-y)\d\rho(x)\d\rho(y)$.
For the corresponding aggregation equations without diffusion, global existence and uniqueness of weak measure solutions, together with finite-time concentration, were established in \cite{carrillo2011global}. The structure of local minimizers of interaction energies, particularly the dimensionality of their supports, was investigated in \cite{balague2013dimensionality}. For models with diffusion, convergence rates toward equilibrium were obtained through entropy dissipation and transportation inequalities in \cite{carrillo2003kinetic}, while Wasserstein contraction methods were developed in \cite{carrillo2006contractions}. Uniform exponential convergence was further established for certain nonconvex interaction or confinement potentials in \cite{bolley2012uniform}. The equilibrium structure can also change with the relative strengths of diffusion and interaction, giving rise to bifurcations of stationary states and continuous or discontinuous phase transitions, as shown for McKean--Vlasov equations on the torus in \cite{carrillo2020phase}.

Beyond pairwise interaction energies, simplex volumes provide geometric
quantities involving several sampled points. The volume of a simplex
reflects the joint dispersion of its vertices, and varying its dimension
leads to a family of nonlocal interaction functionals. For
$1\leq n\leq d$ and $x_0,\ldots,x_n\in\mathbb R^d$, we define the edge
matrix and the corresponding simplex volume by
\[
\begin{aligned}
B(x_0,\dots,x_n)
&:=
\begin{bmatrix}
 x_1-x_0 & \cdots & x_n-x_0
\end{bmatrix}
\in\bbr^{d\times n},\\
\mathrm{Vol}_n(x_0,\dots,x_n)
&:=\frac{1}{n!}\sqrt{\det(B^\top B)}.
\end{aligned}
\]
For $\rho\in\mathcal P_2(\bbr^d)$, the extended generalized variance is
defined by
\begin{equation}
\label{eq:vn}
\mathcal V_n[\rho]
:=
\int_{(\mathbb R^d)^{n+1}}
\mathrm{Vol}_n(x_0,\dots,x_n)^2
\,\d\rho^{\otimes(n+1)}(x_0,\dots,x_n).
\end{equation}
The functional $\mathcal V_n[\rho]$ is the expected squared volume of an
$n$-simplex whose vertices are independently sampled from $\rho$.
It belongs to the family of multivariate dispersion measures introduced
in \cite{pronzato2017extended}. At $n=1$, it is the mean squared
pairwise distance, whereas $n=d$ gives, up to a constant, the classical
generalized variance \cite{wilks1932certain}. The intermediate orders
$1<n<d$ describe dispersion in lower-dimensional directions.
Related simplicial functionals arise in generalized Mahalanobis
distances and distances between probability distributions
\cite{pronzato2018simplicial}, as well as in optimal experimental design
and Bregman-type divergences
\cite{pronzato2017extended,pronzato2019bregman}.

To investigate how the simplex dimension affects the equilibrium structure
and long-time behavior in the presence of diffusion, we consider the
nonlocal Fokker--Planck equation
\begin{equation}
\label{PDE}
\partial_t\rho_t+\nabla\cdot(\rho_t v_t)
=\gamma\Delta\rho_t,
\qquad
v_t(x):=v[\rho_t](x),
\qquad
v[\rho]:=-\nabla\frac{\delta\mathcal V_n}{\delta\rho},
\end{equation}
where $\gamma>0$ is the diffusion coefficient. The equation has a formal
Wasserstein gradient-flow interpretation through the free energy
\begin{equation}
\label{eqn:energy}
\mathcal E[\rho]
:=\gamma\operatorname{Ent}(\rho)+\mathcal V_n[\rho],
\qquad
\rho\in\mathcal P_2(\bbr^d),
\end{equation}
where the entropy is defined by
\[
\operatorname{Ent}(\rho)
:=
\begin{cases}
\displaystyle
\int_{\bbr^d}f(x)\log f(x)\,\d x,
& \rho=f\,\mathcal L^d,\\[1ex]
+\infty,
& \text{otherwise},
\end{cases}
\]
with $\mathcal L^d$ denoting Lebesgue measure and $0\log0:=0$.  The Wasserstein gradient-flow formulation provides a general framework for diffusion and interaction equations \cite{AGS2005,JKO,otto2001geometry,santambrogio2017euclidean}. In the present setting, the first variation of the entropy at a smooth positive density $f$ is $1+\log f$, so the velocity associated with $\mathcal E$ is $v[\rho]-\gamma\nabla\log f$. The corresponding continuity equation then recovers \eqref{PDE}.

Equation \eqref{PDE} also admits a particle interpretation. Since
$\mathrm{Vol}_n^2$ is symmetric in its arguments, the first variation of
$\mathcal V_n$ at $x$ is, up to an additive constant, the integral of
$(n+1)\mathrm{Vol}_n(x,x_1,\ldots,x_n)^2$ with respect to
$\rho^{\otimes n}$, and hence, as observed in \cite{kim2026finite},
\[
v[\rho](x)
=-(n+1)\int_{(\bbr^d)^n}
\nabla_x\bigl[\mathrm{Vol}_n(x,x_1,\ldots,x_n)^2\bigr]
\,\d\rho^{\otimes n}(x_1,\ldots,x_n).
\]
Evaluating this expression at the empirical measure of $N$ particles
shows that \eqref{PDE} is the formal mean-field equation of the system
\[
\d X_t^i
=-\frac{n+1}{N^n}\sum_{j_1,\ldots,j_n=1}^N
\nabla_x\bigl[\mathrm{Vol}_n(x,X_t^{j_1},\ldots,X_t^{j_n})^2\bigr]_{x=X_t^i}
\,\d t
+\sqrt{2\gamma}\,\d B_t^i,
\qquad 1\leq i\leq N,
\]
where $B^1,\ldots,B^N$ are independent Brownian motions. Each particle
thus interacts with $n$ others through the squared volume of the simplex
they span, and the interaction is pairwise only when $n=1$. The rigorous
derivation of \eqref{PDE} from this system is not addressed here.

The nondiffusive evolution generated by $\mathcal V_n$, corresponding to
$\gamma=0$, was studied in \cite{kim2026finite}. In that setting, every
nonstationary solution approaches a limiting covariance of rank strictly
less than $n$. The rank-deficient limit reflects the concentration
induced by the interaction. For $\gamma>0$, diffusion acts against concentration, and the
long-time balance between the two mechanisms is no longer apparent.
Moreover, $\mathcal V_n$ is built from $n$-dimensional simplex volumes in
$\bbr^d$, making the relation between $n$ and $d$ relevant to the possible
covariance shapes. Even after the stationary states have been identified,
it remains to determine which one is approached from a given initial
measure and how rapidly convergence takes place.

A distinctive feature of \eqref{PDE} is that the interaction energy
depends on $\rho$ only through its covariance matrix. More precisely,
$\mathcal V_n[\rho]$ is a constant multiple of the $n$-th elementary
symmetric polynomial of the covariance eigenvalues
\cite{pronzato2017extended}, and the resulting velocity field is affine
in position, with coefficients determined by the mean and covariance of
$\rho$ (see Section~\ref{sec:preliminaries}). The nonlinearity therefore
enters \eqref{PDE} only through finitely many moments. These moments
satisfy a closed system of ordinary differential equations, and along a
given covariance trajectory the equation reduces to a linear
Fokker--Planck equation with a time-dependent linear drift. This
structure makes the model explicitly tractable.

The interest of the model lies in the fact that, despite this structure,
it falls outside the standard convexity framework. At $n=1$, the
functional $\mathcal V_1$ is the pairwise interaction energy with
quadratic potential, and \eqref{PDE} is covered by the theory of convex
interaction potentials in
\cite{carrillo2003kinetic,carrillo2006contractions}. For $n\geq2$, the
situation is different in two respects. First, the free energy
$\mathcal E$ is not $\lambda$-convex along Wasserstein geodesics for any
$\lambda\in\bbr$ (see Proposition~\ref{prop:nonconvexity}), so that
contraction estimates based on displacement convexity
\cite{carrillo2006contractions,mccann1997convexity} are not available.
Second, at $n=d$ the equilibria with a prescribed mean form a noncompact
family, and neither the limit of a given solution nor its convergence
rate can be read off from the free energy alone. The model thus provides
a setting in which equilibrium selection and sharp convergence rates can
be determined explicitly for a nonconvex free energy.

The reduction to the covariance is shared with the nondiffusive model
of \cite{kim2026finite}. The two problems differ, however, in their
equilibria and in the mechanisms that govern convergence. Without
diffusion, the solution is a linear pushforward of the initial measure,
the limit is a degenerate measure, and the convergence rates, which may
be exponential or algebraic, are determined by the covariance
eigenvalues alone. With diffusion, the limits are nondegenerate Gaussian
minimizers of the free energy, which form a continuum at $n=d$ and are
selected through the conserved eigenvalue gaps of the covariance.
Moreover, the solution is the convolution of a transported copy of the
initial measure with a Gaussian measure. Its convergence is always
exponential, with explicit exponents, but the rate is no
longer determined by the covariance alone, since the non-Gaussian part of the solution relaxes at its own
rate. This second mechanism produces the regimes of
Theorem~\ref{thm:main-third-order} below, in which the third moments of
the initial datum dictate the long-time behavior.

Nonlinear Fokker--Planck equations whose coefficients depend on the
covariance of the solution also arise as mean-field limits of ensemble
Kalman methods. For the ensemble Kalman sampler introduced in
\cite{garbuno2020interacting}, the mean and covariance satisfy closed
equations when the forward model is linear, and Wasserstein stability
estimates for the corresponding covariance-preconditioned equation were
obtained in \cite{carrillo2021wasserstein}. The associated
covariance-modulated transport geometry and its gradient flows, in which
the evolution splits into that of the moments and that of the shape of
the distribution, were analyzed in \cite{burger2025covariance}. In these
works the covariance acts as a preconditioner, and the equilibrium is a
prescribed target measure. In \eqref{PDE}, by contrast, the covariance
enters through the interaction energy, the metric is the standard
quadratic Wasserstein metric, and the equilibrium is determined by the
dynamics. After the covariance reduction, the remaining equation is a
linear Fokker--Planck equation, whose convergence to equilibrium in the
autonomous case is classical \cite{arnold2001convex}. For McKean--Vlasov
equations with several stationary solutions, local convergence rates
near stable stationary solutions were recently obtained in
\cite{monmarche2025local}. The present model complements these results
with an example in which the stationary solutions form a continuum and
the convergence is global, with explicit rates.

These considerations lead to the following questions.
\begin{quote}
\begin{enumerate}
\item How does the competition between diffusion and the simplex-volume
interaction determine the stationary states, and what changes between
$n<d$ and $n=d$?
\item When multiple equilibria exist, which one is selected by a given
initial measure?
\item Do solutions converge to equilibrium, and which mechanism
determines the sharp convergence rate of the full probability
measure?
\end{enumerate}
\end{quote}

To address these questions, we first establish a global solution theory.
The dependence of the interaction drift on the moments of the evolving
measure also leads to a closed system for the covariance.

\begin{theorem}[Global well-posedness and covariance dynamics]
\label{thm:main-wellposedness}
Let $\rho_0\in\mathcal P_2(\bbr^d)$ and $\gamma>0$.
Then equation \eqref{PDE} admits a unique global measure-valued weak
solution. In addition, its mean is conserved, while its covariance
matrix evolves according to a closed autonomous finite-dimensional
system and becomes positive definite for every $t>0$. Moreover,
$\rho_t$ has a smooth strictly positive density for every $t>0$.
\end{theorem}

The proof exploits the affine structure of the drift described
above. The covariance equation preserves the initial eigenspaces and
the ordering of eigenvalues, allowing the matrix dynamics to be reduced
to scalar equations (see Propositions~\ref{prop:moment-equations-diffusion}
and \ref{prop:global-reduced-diffusion}). Once the covariance trajectory
is known, the nonlinear equation becomes a linear Fokker--Planck
equation. Its solution is an affine pushforward of the initial measure
convolved with a Gaussian, separating the transported initial
distribution from the spreading induced by diffusion (see
Theorem~\ref{thm:global-wellposedness-diffusion}). This representation
also leads to Wasserstein bounds used in the long-time analysis (see
Proposition~\ref{prop:W2-bounds-diffusion}).

Using the covariance reduction and the explicit representation of solutions, the classification of equilibria and their long-time behavior are given as follows.

\begin{theorem}[Equilibria and long-time behavior]
\label{thm:main-longtime}
The following assertions hold.

\begin{enumerate}
\item The stationary weak solutions of \eqref{PDE}
coincide with the global minimizers of $\mathcal E$.
For $1\leq n<d$, these are precisely the Gaussian
measures $\mathcal N(m,\lambda_\ast I_d)$ with
$m\in\bbr^d$, where
\[
\lambda_\ast
=\left(
\frac{\gamma n!}
{2(n+1)\binom{d-1}{n-1}}
\right)^{1/n}.
\]
For $n=d$, they are precisely the Gaussian measures
$\mathcal N(m,\Sigma)$ with $m\in\bbr^d$ and
$\Sigma$ positive definite satisfying
\[
\det\Sigma=\frac{\gamma d!}{2(d+1)}.
\]

\item For every $\rho_0\in\mathcal P_2(\bbr^d)$,
the solution from
Theorem~\ref{thm:main-wellposedness}
converges exponentially in $W_2$ to an equilibrium
$\rho_\infty=\mathcal N(m_0,\Sigma_\infty)$,
where $m_0$ and $\Sigma_0$ are the mean and
covariance of $\rho_0$, respectively.
The limiting covariance is given by
\[
\Sigma_\infty=
\begin{cases}
\lambda_\ast I_d, &1\leq n<d,\\
\Sigma_0+\overline\mu I_d, &n=d,
\end{cases}
\]
where $\overline\mu$ is the unique real number
such that $\Sigma_0+\overline\mu I_d$ is positive
definite and
\[
\det(\Sigma_0+\overline\mu I_d)
=\frac{\gamma d!}{2(d+1)}.
\]
More precisely, there exist constants $c,C>0$
such that
\[
W_2(\rho_t,\rho_\infty)\leq Ce^{-ct},
\qquad t\geq0.
\]

\item For $1\leq n<d$, suppose that
$\rho_0\in\mathcal P_4(\bbr^d)$ has
non-isotropic covariance. The convergence in (2) then satisfies the following sharp bounds for some constants $A,C>0$:
\[
\begin{aligned}
A\exp\left(
-\frac{2\gamma(d-n)}
{(d-1)\lambda_\ast}t
\right)\leq W_2(\rho_t,\rho_\infty)\leq C\exp\left(
-\frac{2\gamma(d-n)}
{(d-1)\lambda_\ast}t
\right),
\qquad t\geq0.
\end{aligned}
\]
\end{enumerate}
\end{theorem}

The equilibrium classification is based on a reduction to Gaussian
measures. At fixed mean and covariance, the simplex-volume interaction
is unchanged, whereas the Gaussian uniquely minimizes the entropy (see
Proposition~\ref{prop:Gaussian-reduction-minimizer}). The remaining
problem is therefore to minimize over covariance matrices. When $n<d$,
balancing unequal covariance eigenvalues while preserving their product
strictly decreases the interaction energy, forcing the minimizing
covariance to be isotropic. For $n=d$, the Gaussian free energy depends
on the covariance only through its determinant, allowing anisotropic
minimizers. These minimizers are precisely the stationary weak solutions
(see Proposition~\ref{prop:minimizers-stationary}).

The long-time covariance dynamics exhibit the same distinction between the two regimes. For $n=d$, the eigenvalue differences remain constant, whereas for $1\leq n<d$, they decay exponentially (see Proposition~\ref{prop:covariance-nd}, Lemma~\ref{lem:uniform-bounds-synchronization}, and Proposition~\ref{prop:covariance-nlessd}). The convergence rates further distinguish changes in the overall size of the covariance from changes in its shape. When $1<n<d$, the isotropic mode decays faster than the anisotropic one (see Corollary~\ref{cor:sharp-covariance-rates}).
For the full measure-valued solution, the explicit representation separates covariance relaxation from the contribution of the transported initial distribution. These two mechanisms may decay at different rates, and their competition determines the available Wasserstein estimates. The Gaussian convolution permits a stronger control of the transported contribution when the initial measure has a finite fourth moment, leading to matching upper and lower bounds in the subcritical case with non-isotropic initial covariance (see Theorem~\ref{thm:explicit-W2-decay}).

The bounds in Theorem~\ref{thm:main-longtime} do not identify the sharp
rate when covariance relaxation is faster than the decay of the
transported initial distribution. This occurs at $n=d\geq2$ and, for
$2\leq n<d$, when the initial covariance is isotropic. Our third result
identifies the rate in these regimes for compactly supported initial
data under a condition on the third moments. To state it, let $\rho_\infty=\mathcal N(m_0,\Sigma_\infty)$ be the
limit in Theorem~\ref{thm:main-longtime}, let $E_*\subset\bbr^d$ be the
eigenspace of $\Sigma_\infty$ associated with its largest eigenvalue
$\lambda_{\max}(\Sigma_\infty)$, and set
\[
\beta_*:=\frac{\gamma}{\lambda_{\max}(\Sigma_\infty)}.
\]
The exponent $\beta_*$ is the slowest rate at which the linear part of
the flow contracts the initial distribution. When
$\Sigma_0\neq\Sigma_\infty$, we denote by $\kappa$ the sharp exponent of
covariance relaxation, given explicitly in
\eqref{eq:covariance-decay-rate}.

\begin{theorem}[Sharp rates for compactly supported data]
\label{thm:main-third-order}
Let $\rho_0$ be a compactly supported probability measure on
$\bbr^d$ with $\Sigma_0\neq\Sigma_\infty$. Then the following assertions hold.

\begin{enumerate}
\item There exists a constant $C>0$ such that
\[
W_2(\rho_t,\rho_\infty)
\leq C\bigl(e^{-3\beta_*t}+e^{-\kappa t}\bigr),
\qquad t\geq0.
\]

\item Suppose, in addition, that
\[
\int_{\bbr^d}\bigl(\xi\cdot(x-m_0)\bigr)^3\,\d\rho_0(x)\neq0
\qquad\text{for some }\xi\in E_*.
\]
Then there exist constants $A,C>0$ such that
\[
Ae^{-\min\{\kappa,3\beta_*\}t}
\leq W_2(\rho_t,\rho_\infty)
\leq Ce^{-\min\{\kappa,3\beta_*\}t},
\qquad t\geq0.
\]

\item If $n=d\geq2$, or if $2\leq n<d$ and $\Sigma_0$ is isotropic,
then $\kappa>3\beta_*$. In these cases, the sharp rate in (2) is
$3\beta_*$, which is the decay rate of the third moments of the
solution rather than the covariance relaxation rate.
\end{enumerate}
\end{theorem}

The proof of Theorem~\ref{thm:main-third-order} rests on two
ingredients. The upper bound follows from a third-order Gaussian
smoothing estimate. For a bounded centered perturbation of size
$\varepsilon$, the distance between its Gaussian convolution and the
Gaussian measure with the same covariance is of order $\varepsilon^3$
(see Lemma~\ref{lem:third-order-smoothing}). The proof combines an
identity expressing a weighted $L^2$ distance between the two densities
through the differences of their moment tensors with Talagrand's
transportation inequality \cite{talagrand1996transportation}. The lower
bound uses the fact that the third moments of the solution are
transported exactly by the linear part of the flow, while their
differences are controlled by the Wasserstein distance under uniform
fourth-moment bounds (see Lemma~\ref{lem:third-moment-lower}). In the
regimes described in part (3), the long-time behavior is therefore
governed by the skewness of the initial datum, a mechanism that is
absent from the covariance dynamics and from the nondiffusive model
(see Theorem~\ref{thm:third-order-rates} and the numerical
illustration in Section~\ref{subsec:numerics}).

The rest of the paper is organized as follows.
Section~\ref{sec:preliminaries} introduces the covariance identities
and Newton transforms underlying the interaction drift and recalls
the nondiffusive model.
Section~\ref{sec:wellposedness} develops the reduced covariance system,
the representation of measure-valued solutions, and the well-posedness
and Wasserstein estimates.
Section~\ref{sec:minimizers} studies the free-energy minimization
problem, identifies the stationary states,
and shows that the free energy is not displacement
semiconvex.
Section~\ref{sec:longtime} analyzes equilibrium selection, obtains
quantitative rates for covariance and Wasserstein convergence,
identifies the sharp rates for compactly supported initial
data, and presents a numerical illustration.
The appendices provide the cutoff and duality arguments,
together with the Gaussian smoothing estimates used in
the Wasserstein convergence analysis.

\section{Preliminaries}
\label{sec:preliminaries}

\subsection{Moments and Newton transforms}
For $\rho\in\mathcal P_2(\bbr^\dm)$, we write
\[
m(\rho):=\int_{\bbr^\dm}x\,\d\rho(x),
\qquad
\Sigma(\rho)
:=\int_{\bbr^\dm}
(x-m(\rho))(x-m(\rho))^\top\,\d\rho(x),
\]
and if $(\rho_t)_{t\ge 0}$ is a curve in $\mathcal P_2(\mathbb R^d)$, we write
$m_t := m(\rho_t)$ and $\Sigma_t := \Sigma(\rho_t)$ for simplicity.
Let $A$ be a symmetric $\dm\times\dm$ matrix with eigenvalues
$\lambda_1,\ldots,\lambda_\dm$. We denote by $e_k(A)$ the $k$-th elementary
symmetric polynomial of these eigenvalues and set $e_0(A):=1$. When one or two
eigenvalues are omitted, we write $e_k(\hat\lambda_i)$ and
$e_k(\hat\lambda_i,\hat\lambda_j)$, respectively. We use the convention
that an elementary symmetric polynomial is zero when its degree is negative or
exceeds the number of available variables.

The simplex-volume interaction admits the covariance representation established in
\cite{pronzato2017extended} and revisited from an exterior-algebra viewpoint in
\cite{WZ}, given by
\[
\mathcal V_n[\rho]
=c_n e_n(\Sigma(\rho)),
\qquad
c_n:=\frac{n+1}{n!}.
\]
The identity follows by centering the sampled vertices and exploiting their independence, which leaves only second moments in the expectation of the squared determinant.

To express the first variation of this finite-dimensional quantity, we use the
$k$-th Newton transform
\[
T_k(A):=\sum_{\ell=0}^k(-1)^\ell e_{k-\ell}(A)A^\ell.
\]
If $A$ is diagonal with diagonal entries $\lambda_1,\ldots,\lambda_\dm$, then
\[
T_k(A)
=\operatorname{diag}\bigl(
 e_k(\hat\lambda_1),\ldots,e_k(\hat\lambda_\dm)
\bigr).
\]
This follows directly from the recurrence
$e_j(\lambda_1,\ldots,\lambda_\dm)
=e_j(\hat\lambda_i)+\lambda_i e_{j-1}(\hat\lambda_i)$ applied in the
$i$-th diagonal entry of the defining polynomial for $T_k$.

The first variation can be computed from the covariance identity. The differential of the $n$-th elementary symmetric polynomial satisfies
\[
D e_n(\Sigma)[H]
=\operatorname{tr}\bigl(T_{n-1}(\Sigma)H\bigr).
\]
Fix $\nu\in\mathcal P_2(\bbr^\dm)$ and consider the probability measures $\rho_\varepsilon:=(1-\varepsilon)\rho+\varepsilon\nu$ for $0\leq\varepsilon\leq1$. Writing $m:=m(\rho)$ and $\Sigma:=\Sigma(\rho)$, differentiation at $\varepsilon=0$ gives
\[
\left.\frac{\d}{\d\varepsilon}\right|_{\varepsilon=0^+}
\Sigma(\rho_\varepsilon)
=\int_{\bbr^\dm}(x-m)(x-m)^\top\,\d(\nu-\rho)(x).
\]
Applying the differential identity above, we obtain
\[
\left.\frac{\d}{\d\varepsilon}\right|_{\varepsilon=0^+}
\mathcal V_n[\rho_\varepsilon]
=c_n\int_{\bbr^\dm}
(x-m)^\top T_{n-1}(\Sigma)(x-m)\,\d(\nu-\rho)(x).
\]
Since the first variation is defined up to an additive constant, this yields
\[
\frac{\delta\mathcal V_n}{\delta\rho}(x)
=c_n(x-m)^\top T_{n-1}(\Sigma)(x-m).
\]
Taking its spatial gradient gives the drift in \eqref{PDE}:
\[
v[\rho](x)
=-\nabla\frac{\delta\mathcal V_n}{\delta\rho}(x)
=-2c_nT_{n-1}(\Sigma(\rho))(x-m(\rho)).
\]

\subsection{The nondiffusive model}
Setting $\gamma=0$ in \eqref{PDE} gives the nondiffusive Wasserstein gradient flow studied in \cite{kim2026finite},
\[
\partial_t\rho_t+\nabla\cdot(\rho_t v[\rho_t])=0,
\qquad
v[\rho](x)=-2c_nT_{n-1}(\Sigma(\rho))(x-m(\rho)),
\]
whose covariance satisfies
\[
\frac{\d}{\d t}\Sigma_t
=-4c_n\Sigma_tT_{n-1}(\Sigma_t).
\]
The full solution is obtained by a linear pushforward of the initial measure, and every nonstationary trajectory approaches a limiting covariance of rank strictly smaller than $n$ \cite{kim2026finite}. With diffusion, the additional forcing $2\gamma I_\dm$ makes the covariance positive definite for every positive time and introduces a nondegenerate Gaussian factor into the representation derived in Section~\ref{sec:wellposedness}.

\section{Global well-posedness and explicit representation}
\label{sec:wellposedness}
\setcounter{equation}{0}

We first derive the closed evolution equations for the mean and covariance. These equations determine the time-dependent drift, which is then used to construct the measure-valued solution and establish its uniqueness.

Throughout the paper, we use the following notion of weak solution.

\begin{definition}
\label{def:weak-solution-diffusion}
Let $T>0$ and $\rho_0\in\mathcal P_2(\bbr^\dm)$. A curve
$(\rho_t)_{t\in[0,T]}\subset\mathcal P_2(\bbr^\dm)$ is called a weak solution to
\eqref{PDE} on $[0,T]$ with initial datum $\rho_0$ if the map $t\longmapsto\rho_t$ is continuous from $[0,T]$ into
$\bigl(\mathcal P_2(\bbr^\dm),W_2\bigr)$ and
\[
\begin{aligned}
&\int_0^T\int_{\bbr^\dm}
\left(
\partial_t\varphi(t,x)
+v[\rho_t](x)\cdot\nabla\varphi(t,x)
+\gamma\Delta\varphi(t,x)
\right)
\,\d\rho_t(x)\,\d t
+\int_{\bbr^\dm}\varphi(0,x)\,\d\rho_0(x)=0
\end{aligned}
\]
for every $\varphi\in C_c^\infty([0,T)\times\bbr^\dm)$.
\end{definition}

\subsection{Moment equations and reduced dynamics}

We begin by showing that the first two moments of every weak solution satisfy a closed system. We present the formal moment calculation here and justify the use of unbounded test functions by cutoff arguments in Appendix~\ref{app:moment-cutoff}.

\begin{proposition}[Moment equations]
\label{prop:moment-equations-diffusion}
Let $(\rho_t)_{t\in[0,T]}$ be a weak solution in the sense of
Definition~\ref{def:weak-solution-diffusion}, and set
\[
m_t:=m(\rho_t),
\qquad
\Sigma_t:=\Sigma(\rho_t).
\]
Then the following assertions hold.
\begin{enumerate}
\item The mean is conserved. More precisely,
\[
m_t=m_0
\qquad
\text{for every }t\in[0,T].
\]

\item The covariance matrix is continuously differentiable on $[0,T]$ and satisfies
\[
\frac{\d}{\d t}\Sigma_t
=-4c_n\Sigma_tT_{n-1}(\Sigma_t)+2\gamma I_\dm
\]
for every $t\in[0,T]$.
\end{enumerate}
\end{proposition}

\begin{proof}
Applying the moment identity justified in Appendix~\ref{app:moment-cutoff} to the mean gives
\[
\frac{\d}{\d t}m_t
=\int_{\bbr^\dm}v[\rho_t](x)\,\d\rho_t(x)=0,
\]
by using $v[\rho_t](x)=-A_t(x-m_t)$ with
$A_t:=2c_nT_{n-1}(\Sigma_t)$ and
$\int_{\bbr^\dm}(x-m_t)\,\d\rho_t=0$. Hence $m_t=m_0$.

Applying the same calculation to the second moment gives
\[
\frac{\d}{\d t}\Sigma_t
=-\Sigma_tA_t-A_t\Sigma_t+2\gamma I_\dm.
\]
Since $T_{n-1}(\Sigma_t)$ is a polynomial in $\Sigma_t$, it commutes with $\Sigma_t$. Then, substituting the definition of $A_t$ gives the claimed covariance equation. The cutoff argument also yields the asserted regularity in time.
\end{proof}

The covariance equation is globally solvable and preserves both the initial eigenspaces and the ordering of the eigenvalues as follows.

\begin{proposition}[Global solvability of the reduced system]
\label{prop:global-reduced-diffusion}
Let $\Sigma_0$ be a symmetric positive semidefinite matrix. Then the following
assertions hold.
\begin{enumerate}
\item The covariance equation
\begin{equation}
\label{eq:covariance-diffusion}
\frac{\d}{\d t}\Sigma_t
=-4c_n\Sigma_tT_{n-1}(\Sigma_t)+2\gamma I_\dm,
\qquad
\Sigma_{t}\big|_{t=0}=\Sigma_0,
\end{equation}
admits a unique global solution. The matrix $\Sigma_t$ is symmetric and positive semidefinite for every $t\geq0$, and it is positive definite for every $t>0$.

\item If
\[
\Sigma_0=P^\top D_0P,
\qquad
D_0=\operatorname{diag}(\lambda_1^0,\ldots,\lambda_\dm^0),
\]
where $P$ is orthogonal and $\lambda_i^0\geq0$, then
\[
\Sigma_t=P^\top D_tP,
\qquad
D_t=\operatorname{diag}(\lambda_1(t),\ldots,\lambda_\dm(t)),
\]
and the eigenvalues satisfy
\begin{equation}
\label{eq:eigenvalue-diffusion}
\frac{\d}{\d t}\lambda_i(t)
=-4c_n\lambda_i(t)e_{n-1}(\hat\lambda_i(t))+2\gamma,
\qquad
1\leq i\leq\dm.
\end{equation}

\item For every pair of distinct indices $i,j\in\{1,\ldots,\dm\}$,
\begin{equation}
\label{eq:eigenvalue-difference-diffusion}
\begin{aligned}
\lambda_i(t)-\lambda_j(t)
={}&(\lambda_i^0-\lambda_j^0)
\times\exp\left(
-4c_n\int_0^t e_{n-1}(\hat\lambda_i(s),\hat\lambda_j(s))\,\d s
\right),
\end{aligned}
\end{equation}
where both $\lambda_i$ and $\lambda_j$ are omitted in the elementary
symmetric polynomial. Consequently, the ordering and multiplicities of the
eigenvalues are preserved.
\end{enumerate}
\end{proposition}

\begin{proof}
First of all, the right-hand side of \eqref{eq:covariance-diffusion} is polynomial in the entries of $\Sigma_t$ and is therefore locally Lipschitz on the finite-dimensional space of symmetric matrices.

Now, consider the eigenvalue system \eqref{eq:eigenvalue-diffusion}. Its vector field is locally Lipschitz on $\bbr^\dm$, so there is a unique maximal solution. While the solution remains in the nonnegative orthant, set
\[
a_i(t):=4c_ne_{n-1}(\hat\lambda_i(t))\geq0.
\]
Accordingly, the $i$-th equation can be written as
\[
\frac{\d}{\d t}\lambda_i(t)+a_i(t)\lambda_i(t)=2\gamma.
\]
Then the variation-of-constants formula yields
\[
\lambda_i(t)
=e^{-\int_0^t a_i(r)\,\d r}\lambda_i^0
+2\gamma\int_0^t e^{-\int_s^t a_i(r)\,\d r}\,\d s,
\]
which shows that  $\lambda_i(t)>0$ for every $t>0$. Hence the nonnegative orthant is positively invariant. Moreover, integrating the bound $\frac{\d}{\d t}\lambda_i(t)\leq2\gamma$ gives
\[
0\leq\lambda_i(t)\leq\lambda_i^0+2\gamma t
\]
on every finite interval. The resulting bound rules out finite-time blow-up and therefore extends the eigenvalue solution globally.

Next, let us define $\widetilde\Sigma_t:=P^\top D_tP$. Since
\[
T_{n-1}(P^\top D_tP)=P^\top T_{n-1}(D_t)P
\]
and the $i$-th diagonal entry of $T_{n-1}(D_t)$ is $e_{n-1}(\hat\lambda_i(t))$, the matrix $\widetilde\Sigma_t$ solves \eqref{eq:covariance-diffusion} with initial datum $\Sigma_0$. Therefore, the uniqueness of the matrix ODE gives $\Sigma_t=\widetilde\Sigma_t$.

For the difference identity \eqref{eq:eigenvalue-difference-diffusion}, fix $i\neq j$. The elementary symmetric polynomials satisfy
\[
\begin{aligned}
e_{n-1}(\hat\lambda_i)
&=e_{n-1}(\hat\lambda_i,\hat\lambda_j)
+\lambda_j e_{n-2}(\hat\lambda_i,\hat\lambda_j),
\\
e_{n-1}(\hat\lambda_j)
&=e_{n-1}(\hat\lambda_i,\hat\lambda_j)
+\lambda_i e_{n-2}(\hat\lambda_i,\hat\lambda_j).
\end{aligned}
\]
These identities imply
\[
\lambda_i e_{n-1}(\hat\lambda_i)
-\lambda_j e_{n-1}(\hat\lambda_j)
=(\lambda_i-\lambda_j)e_{n-1}(\hat\lambda_i,\hat\lambda_j).
\]
Subtracting the equations for $\lambda_i$ and $\lambda_j$, the diffusion terms cancel and we obtain
\[
\frac{\d}{\d t}(\lambda_i-\lambda_j)
=-4c_ne_{n-1}(\hat\lambda_i,\hat\lambda_j)(\lambda_i-\lambda_j).
\]
Finally, integrating this scalar equation gives \eqref{eq:eigenvalue-difference-diffusion}, which shows that the sign of every eigenvalue difference is preserved and that any equality present initially remains valid for all later times.
\end{proof}

\subsection{Explicit construction of the solution}

We now use the uniquely determined covariance trajectory to construct a solution of the full equation. Let $\Sigma_t$ be the global solution of \eqref{eq:covariance-diffusion}, and recall
\begin{equation}
\label{eq:def-At-diffusion}
A_t=2c_nT_{n-1}(\Sigma_t).
\end{equation}
Since $\Sigma_t$ is symmetric and positive semidefinite, Proposition~\ref{prop:global-reduced-diffusion} yields
\[
\Sigma_t=P^\top D_tP,
\qquad
D_t=\operatorname{diag}\bigl(\lambda_1(t),\ldots,\lambda_\dm(t)\bigr),
\]
where $P$ is orthogonal and $\lambda_i(t)\geq0$ for every $i\in\{1,\ldots,\dm\}$. In the same basis, we have
\[
T_{n-1}(D_t)
=\operatorname{diag}\left(
e_{n-1}(\hat\lambda_1(t)),\ldots,e_{n-1}(\hat\lambda_\dm(t))
\right),
\]
and hence
\[
T_{n-1}(\Sigma_t)=P^\top T_{n-1}(D_t)P.
\]
The diagonal representation shows that $A_t$ is symmetric and positive semidefinite. Its continuity follows from the polynomial dependence of $T_{n-1}(\Sigma_t)$ on the continuous covariance path. For each $s\geq0$, the linear system with coefficient $-A_t$ therefore admits a unique fundamental matrix $\Phi(\cdot,s)\in C^1([s,\infty);\bbr^{\dm\times\dm})$ satisfying
\begin{equation}
\label{eq:fundamental-matrix-diffusion}
\partial_t\Phi(t,s)=-A_t\Phi(t,s),
\qquad
\Phi(s,s)=I_\dm.
\end{equation}
By Liouville's formula,
\[
\det\Phi(t,s)
=
\exp\left(
-\int_s^t\operatorname{tr}(A_r)\,\d r
\right)>0,
\qquad
0\leq s\leq t.
\]
Hence $\Phi(t,s)$ is invertible for every $0\leq s\leq t$.

Now, define
\begin{equation}
\label{eq:def-Qts-diffusion}
Q(t,s)
:=
2\gamma
\int_s^t
\Phi(t,r)\Phi(t,r)^\top\,\d r,
\qquad
Q_t:=Q(t,0).
\end{equation}
Then the matrix $Q(t,s)$ is symmetric and positive semidefinite for every $0\leq s\leq t$. Moreover, it is positive definite whenever $t>s$. Indeed, for every $\xi\in\bbr^\dm\setminus\{0\}$,
\[ \begin{aligned} \xi^\top Q(t,s)\xi &= 2\gamma \int_s^t \xi^\top \Phi(t,r)\Phi(t,r)^\top \xi\,\d r \\
&= 2\gamma \int_s^t |\Phi(t,r)^\top\xi|^2\,\d r,\end{aligned} \]
and since $\Phi(t,r)$ is invertible throughout $[s,t]$, the integral is strictly positive for $t>s$.

For an initial measure $\rho_0\in\mathcal P_2(\bbr^\dm)$ with mean $m_0:=m(\rho_0)$, define
\begin{equation}
\label{eq:explicit-solution-diffusion}
\rho_t
:=\bigl(m_0+\Phi(t,0)(\,\cdot\,-m_0)\bigr)_\#\rho_0 * \mathcal N(0,Q_t).
\end{equation}
Equivalently, let $X_0$ be a random variable with law $\rho_0$, and let $Z_t$ be an
independent Gaussian random variable with mean zero and covariance matrix $Q_t$. Then
$\rho_t$ is the law of $m_0+\Phi(t,0)(X_0-m_0)+Z_t$.
The next proposition verifies that the mean and covariance of the measure defined in \eqref{eq:explicit-solution-diffusion} coincide with the prescribed reduced dynamics.

\begin{proposition}[Recovery of the prescribed moments]
\label{prop:moment-recovery-diffusion}
Let $(\rho_t)_{t\geq0}\subset\mathcal P_2(\bbr^\dm)$ be the curve defined by \eqref{eq:explicit-solution-diffusion}. Then
\[
m(\rho_t)=m_0,
\qquad
\Sigma(\rho_t)=\Sigma_t
\]
for every $t\geq0$.
\end{proposition}

\begin{proof}
Since the Gaussian measure $\mathcal N(0,Q_t)$ has mean zero,
the mean of $\rho_t$ is
\[
m(\rho_t)
=\int_{\bbr^\dm}
\left(m_0+\Phi(t,0)(x-m_0)\right)\,\d\rho_0(x)
=m_0.
\]
Since the representation involves the sum of two independent
random variables, their covariance matrices add, giving
\begin{equation}
\label{eq:covariance-representation-diffusion}
\Sigma(\rho_t)
=\Phi(t,0)\Sigma_0\Phi(t,0)^\top+Q_t.
\end{equation}
Differentiating the first term in
\eqref{eq:covariance-representation-diffusion} and using
\eqref{eq:fundamental-matrix-diffusion}, we obtain
\[
\frac{\d}{\d t}\left(\Phi(t,0)\Sigma_0\Phi(t,0)^\top\right)
=-A_t\Phi(t,0)\Sigma_0\Phi(t,0)^\top
-\Phi(t,0)\Sigma_0\Phi(t,0)^\top A_t.
\]
Differentiating \eqref{eq:def-Qts-diffusion} gives
\[
\begin{aligned}
\frac{\d}{\d t}Q_t
&=2\gamma I_\dm
+2\gamma\int_0^t
\partial_t\left(\Phi(t,s)\Phi(t,s)^\top\right)\,\d s\\
&=2\gamma I_\dm-A_tQ_t-Q_tA_t.
\end{aligned}
\]
Combining these identities, we find
\[
\frac{\d}{\d t}\Sigma(\rho_t)
=-A_t\Sigma(\rho_t)-\Sigma(\rho_t)A_t+2\gamma I_\dm.
\]
Since $A_t=2c_nT_{n-1}(\Sigma_t)$ commutes with $\Sigma_t$,
equation \eqref{eq:covariance-diffusion} can also be written as
\[
\frac{\d}{\d t}\Sigma_t
=-A_t\Sigma_t-\Sigma_tA_t+2\gamma I_\dm.
\]
Thus $\Sigma(\rho_t)$ and $\Sigma_t$ satisfy the same linear
matrix ODE with the same initial condition
$\Sigma(\rho_0)=\Sigma_0$.
Uniqueness gives $\Sigma(\rho_t)=\Sigma_t$ for every $t\geq0$.
\end{proof}

The preceding proposition shows that the moments of the constructed measure agree with those used to define its drift. We next verify that this measure satisfies the nonlinear weak equation.

\begin{proposition}[Weak solution property]
\label{prop:weak-solution-explicit-diffusion}
Let $(\rho_t)_{t\geq0}\subset\mathcal P_2(\bbr^\dm)$ be the curve defined by \eqref{eq:explicit-solution-diffusion}. Then $(\rho_t)_{t\geq0}$ belongs to
\[
C\bigl([0,\infty);(\mathcal P_2(\bbr^\dm),W_2)\bigr)
\]
and is a weak solution to \eqref{PDE}.
\end{proposition}

\begin{proof}
Let $X_0$ be a random variable with law $\rho_0$, and let $(B_t)_{t\geq0}$ be a standard $\dm$-dimensional Brownian motion independent of $X_0$. Define
\begin{equation}
\label{eq:stochastic-representation-diffusion}
Y_t
:=m_0+\Phi(t,0)(X_0-m_0)
+\sqrt{2\gamma}\int_0^t\Phi(t,s)\,\d B_s.
\end{equation}
The stochastic integral is Gaussian with mean zero and covariance $Q_t$, and therefore $Y_t$ has law $\rho_t$. Moreover, the variation-of-constants formula gives
\[
\d Y_t=-A_t(Y_t-m_0)\,\d t+\sqrt{2\gamma}\,\d B_t,
\qquad
Y_0=X_0.
\]
For given $T>0$ and $\varphi\in C_c^\infty([0,T)\times\bbr^\dm)$, It\^o's formula gives
\[
\d\varphi(t,Y_t)
=\big[
\partial_t\varphi(t,Y_t)
-A_t(Y_t-m_0)\cdot\nabla\varphi(t,Y_t)
+\gamma\Delta\varphi(t,Y_t)
\big]\,\d t
+\sqrt{2\gamma}\nabla\varphi(t,Y_t)\cdot\d B_t.
\]
Taking expectations and integrating over $[0,T]$, the martingale term disappears and we obtain
\[
\int_0^T\int_{\bbr^\dm}
\big[
\partial_t\varphi(t,x)
-A_t(x-m_0)\cdot\nabla\varphi(t,x)
+\gamma\Delta\varphi(t,x)
\big]
\,\d\rho_t(x)\,\d t
+\int_{\bbr^\dm}\varphi(0,x)\,\d\rho_0(x)=0.
\]
Then by Proposition~\ref{prop:moment-recovery-diffusion}, we have
\[
-A_t(x-m_0)
=-2c_nT_{n-1}(\Sigma(\rho_t))(x-m(\rho_t))
=v[\rho_t](x),
\]
which gives the nonlinear weak formulation.

To check continuity in $W_2$, note that the deterministic part of \eqref{eq:stochastic-representation-diffusion} is continuous in $L^2$ because $t\mapsto\Phi(t,0)$ is continuous. The stochastic convolution is also continuous in $L^2$ by It\^o's isometry and the continuity of $\Phi(t,s)$. Hence
\[
\lim_{s\to t}\mathbb E|Y_s-Y_t|^2=0.
\]
The joint law of $(Y_s,Y_t)$ is an admissible coupling of $\rho_s$ and $\rho_t$. Since $W_2^2$ is the infimum of the quadratic transportation cost over all such couplings, this particular coupling gives
\[
W_2^2(\rho_s,\rho_t)\leq\mathbb E|Y_s-Y_t|^2\longrightarrow0,
\]
which proves the claimed continuity.
\end{proof}

\subsection{Uniqueness of weak solutions}

Propositions~\ref{prop:global-reduced-diffusion},
\ref{prop:moment-recovery-diffusion}, and
\ref{prop:weak-solution-explicit-diffusion} establish the existence of
a global weak solution with the explicit representation
\eqref{eq:explicit-solution-diffusion}. Since $Q_t$ is positive
definite for every $t>0$, this solution also has a smooth strictly
positive density at positive times.

In this subsection, we prove the uniqueness of the global weak solution and derive Theorem~\ref{thm:main-wellposedness}.

\begin{theorem}[Global well-posedness and explicit representation]
\label{thm:global-wellposedness-diffusion}
Let $\rho_0\in\mathcal P_2(\bbr^\dm)$. Then the following assertions hold.
\begin{enumerate}
\item Equation~\eqref{PDE} admits a unique global weak solution
\[
\rho\in C\bigl([0,\infty);(\mathcal P_2(\bbr^\dm),W_2)\bigr).
\]

\item The solution is explicitly represented by
\[
\rho_t
=\left(m_0+\Phi(t,0)(\cdot-m_0)\right)_\#\rho_0
*\mathcal N(0,Q_t),
\]
where $\Sigma_t$, $A_t$, $\Phi(t,s)$, and $Q_t$ are defined by
\eqref{eq:covariance-diffusion}, \eqref{eq:def-At-diffusion},
\eqref{eq:fundamental-matrix-diffusion}, and
\eqref{eq:def-Qts-diffusion}, respectively.

\item For every $t>0$, the matrix $Q_t$ is positive definite.
Consequently, $\rho_t$ has a smooth strictly positive density.
\end{enumerate}
\end{theorem}

\begin{proof}
To prove uniqueness, fix $T>0$ and let
$\mu\in C\bigl([0,T];(\mathcal P_2(\bbr^\dm),W_2)\bigr)$
be any weak solution with $\mu_0=\rho_0$.
By Proposition~\ref{prop:moment-equations-diffusion},
the mean of $\mu_t$ is $m_0$, and its covariance satisfies
\eqref{eq:covariance-diffusion} with initial datum $\Sigma_0$.
Uniqueness of the covariance equation in
Proposition~\ref{prop:global-reduced-diffusion} gives
\[
\Sigma(\mu_t)=\Sigma_t
\qquad \text{for every }t\in[0,T].
\]
Consequently, $\mu_t$ solves the linear equation
\begin{equation}
\label{eq:linear-FP-diffusion}
\partial_t\mu_t
=\nabla\cdot\left(A_t(x-m_0)\mu_t\right)+\gamma\Delta\mu_t,
\qquad
\mu_t\big|_{t=0}=\rho_0,
\end{equation}
where $A_t$ is a prescribed continuous matrix-valued function
of time.

We apply the backward duality identity established in
Appendix~\ref{app:linear-duality}.
Fix $t\in(0,T]$ and $\psi\in C_c^\infty(\bbr^\dm)$, and define
\[
u(s,x)
:=\int_{\bbr^\dm}
\psi\left(m_0+\Phi(t,s)(x-m_0)+z\right)
\,\d\mathcal N(0,Q(t,s))(z),
\qquad 0\leq s\leq t.
\]
As shown in Appendix~\ref{app:linear-duality},
this function solves the backward Kolmogorov equation
associated with \eqref{eq:linear-FP-diffusion} and satisfies
the terminal condition $u(t,x)=\psi(x)$.
The duality identity therefore yields
\[
\int_{\bbr^\dm}\psi(x)\,\d\mu_t(x)
=\int_{\bbr^\dm}u(0,x)\,\d\rho_0(x).
\]
By the definition of $u$ and the explicit representation
\eqref{eq:explicit-solution-diffusion}, we have
\[
\begin{aligned}
\int_{\bbr^\dm}u(0,x)\,\d\rho_0(x)
&=\int_{\bbr^\dm}\int_{\bbr^\dm}
\psi\left(m_0+\Phi(t,0)(x-m_0)+z\right)
\,\d\mathcal N(0,Q_t)(z)\,\d\rho_0(x)\\
&=\int_{\bbr^\dm}\psi(x)\,\d\rho_t(x).
\end{aligned}
\]
Since this equality holds for every
$\psi\in C_c^\infty(\bbr^\dm)$, it follows that
$\mu_t=\rho_t$.
As $t\in(0,T]$ and $T>0$ were arbitrary,
uniqueness holds globally.
\end{proof}

\begin{proof}[Proof of Theorem~\ref{thm:main-wellposedness}]
Proposition~\ref{prop:moment-equations-diffusion} gives conservation of the mean and the closed covariance equation. Proposition~\ref{prop:global-reduced-diffusion} proves global solvability of that equation and strict positivity of the covariance for positive times. Theorem~\ref{thm:global-wellposedness-diffusion} then gives the unique global weak solution, the explicit representation, and instantaneous Gaussian regularization.
\end{proof}

\subsection{Representations and Wasserstein estimates}

We express the transport and Gaussian components of the solution in the common eigenbasis of the covariance matrices. This yields the Wasserstein estimates needed in Section~\ref{sec:longtime}. Suppose that
\[
\Sigma_0=P^\top\operatorname{diag}(\lambda_1^0,\ldots,\lambda_\dm^0)P.
\]
Then the covariance and drift matrices remain diagonal in this basis, with
\[
\begin{aligned}
\Sigma_t
&=P^\top\operatorname{diag}(\lambda_1(t),\ldots,\lambda_\dm(t))P,\\
A_t
&=P^\top\operatorname{diag}\left(
2c_ne_{n-1}(\hat\lambda_1(t)),\ldots,
2c_ne_{n-1}(\hat\lambda_\dm(t))
\right)P.
\end{aligned}
\]
The transition matrix is therefore diagonal in the same basis, with
\begin{equation}
\label{eq:Phi-diagonal-diffusion}
\begin{aligned}
\Phi(t,s)
&=P^\top\operatorname{diag}\left(r_1(t,s),\ldots,r_\dm(t,s)\right)P,\\
r_i(t,s)
&:=\exp\left(
-2c_n\int_s^t e_{n-1}(\hat\lambda_i(\tau))\,\d\tau
\right).
\end{aligned}
\end{equation}
In addition, the same common eigenbasis diagonalizes the noise covariance, which satisfies
\begin{equation}
\label{eq:Q-diagonal-diffusion}
\begin{aligned}
&Q_t=P^\top\operatorname{diag}(q_1(t),\ldots,q_\dm(t))P,\\
&q_i(t):=2\gamma\int_0^t r_i(t,s)^2\,\d s.
\end{aligned}
\end{equation}
To obtain a useful relation for $q_i(t)$, we return to the covariance identity \eqref{eq:covariance-representation-diffusion}. By \eqref{eq:Phi-diagonal-diffusion}, the $i$-th diagonal entry of $\Phi(t,0)\Sigma_0\Phi(t,0)^\top$ in the common eigenbasis is $\lambda_i^0r_i(t,0)^2$. The corresponding entries of $\Sigma_t$ and $Q_t$ are $\lambda_i(t)$ and $q_i(t)$, respectively. Comparing the $i$-th diagonal entries in \eqref{eq:covariance-representation-diffusion} yields
\begin{equation}
\label{eq:qi-covariance-diffusion}
q_i(t)=\lambda_i(t)-\lambda_i^0r_i(t,0)^2.
\end{equation}

The next proposition provides upper and lower Wasserstein bounds
for $\rho_t$ relative to Gaussian measures whose covariances are
diagonal in the common eigenbasis fixed above.

\begin{proposition}[Wasserstein upper and lower bounds]
\label{prop:W2-bounds-diffusion}
Let $P$ be the orthogonal matrix diagonalizing
$\Sigma_0$ as above. For any fixed nonnegative
numbers $\overline\lambda_1,\ldots,\overline\lambda_\dm$,
set
\[
\overline\Sigma
=P^\top\operatorname{diag}
(\overline\lambda_1,\ldots,\overline\lambda_\dm)P,
\qquad
\overline\rho:=\mathcal N(m_0,\overline\Sigma).
\]
Then the following bounds hold for every $t\geq0$.
\begin{enumerate}
\item An explicit coupling gives the upper bound
\begin{equation}
\label{eq:direct-W2-upper-diffusion}
W_2^2(\rho_t,\overline\rho)
\leq
\sum_{i=1}^\dm\lambda_i^0r_i(t,0)^2
+\sum_{i=1}^\dm
\left(\sqrt{q_i(t)}
-\sqrt{\overline\lambda_i}\right)^2.
\end{equation}

\item The covariance matrices give the lower bound
\begin{equation}
\label{eq:covariance-W2-lower}
W_2^2(\rho_t,\overline\rho)
\geq
\sum_{i=1}^\dm
\left(\sqrt{\lambda_i(t)}
-\sqrt{\overline\lambda_i}\right)^2.
\end{equation}
\end{enumerate}
\end{proposition}

\begin{proof}
Let $X_0$ have law $\rho_0$, and let $Z$ be a
standard Gaussian random vector independent
of $X_0$. Define
\[
\begin{aligned}
Y_t&:=m_0+\Phi(t,0)(X_0-m_0)+Q_t^{1/2}Z,\\
\overline Y&:=m_0+\overline\Sigma^{1/2}Z.
\end{aligned}
\]
Then $Y_t$ and $\overline Y$ have laws $\rho_t$
and $\overline\rho$, respectively. Hence
\[
W_2^2(\rho_t,\overline\rho)
\leq\mathbb E|Y_t-\overline Y|^2.
\]
Since $X_0-m_0$ and $Z$ are independent and
centered, the cross term vanishes, and
\[
\begin{aligned}
\mathbb E|Y_t-\overline Y|^2
={}&\operatorname{tr}
\left(\Phi(t,0)\Sigma_0\Phi(t,0)^\top\right)\\
&+\left\|Q_t^{1/2}
-\overline\Sigma^{1/2}\right\|_{\mathrm F}^2,
\end{aligned}
\]
where $\|\cdot\|_{\mathrm F}$ denotes the
Frobenius norm. Using the diagonal
representations
\eqref{eq:Phi-diagonal-diffusion} and
\eqref{eq:Q-diagonal-diffusion}, we obtain
\eqref{eq:direct-W2-upper-diffusion}.

For the lower bound, let $(X,Y)$ be any
coupling of $\rho_t$ and $\overline\rho$, and set
\[
U=P(X-m_0),\qquad V=P(Y-m_0).
\]
Writing $U_i$ and $V_i$ for the $i$-th
components of these random vectors, the
common eigenbasis of the covariance matrices
gives
\[
\mathbb E|U_i|^2=\lambda_i(t),
\qquad
\mathbb E|V_i|^2=\overline\lambda_i.
\]
Then, the Cauchy--Schwarz inequality yields
\[
\begin{aligned}
\mathbb E|U_i-V_i|^2
&=\lambda_i(t)+\overline\lambda_i
-2\mathbb E[U_iV_i]\\
&\geq\lambda_i(t)+\overline\lambda_i
-2\sqrt{\lambda_i(t)\overline\lambda_i}\\
&=\left(\sqrt{\lambda_i(t)}
-\sqrt{\overline\lambda_i}\right)^2.
\end{aligned}
\]
Summing over all coordinates and using
the orthogonality of $P$, we obtain
\[
\begin{aligned}
\sum_{i=1}^\dm
\left(\sqrt{\lambda_i(t)}
-\sqrt{\overline\lambda_i}\right)^2
&\leq\sum_{i=1}^\dm
\mathbb E|U_i-V_i|^2\\
&=\mathbb E|U-V|^2\\
&=\mathbb E|X-Y|^2.
\end{aligned}
\]
Since this holds for every coupling of
$\rho_t$ and $\overline\rho$, taking the
infimum proves \eqref{eq:covariance-W2-lower}.
\end{proof}

\section{Equilibria and free-energy minimizers}
\label{sec:minimizers}
\setcounter{equation}{0}

We begin with the minimization of the free energy \eqref{eqn:energy}, and then prove that its minimizers exhaust the stationary weak solutions of \eqref{PDE}. Throughout, a stationary state means a measure whose constant trajectory satisfies the weak formulation.

\subsection{Reduction to Gaussian measures}
\begin{proposition}[Gaussian reduction at fixed moments]
\label{prop:Gaussian-reduction-minimizer}
Suppose that $\rho\in\mathcal P_2(\bbr^\dm)$ has finite free energy, and set
\[
m:=m(\rho),
\qquad
\Sigma:=\Sigma(\rho),
\qquad
G_\rho:=\mathcal N(m,\Sigma).
\]
Then $\Sigma$ is positive definite and
\begin{equation}
\label{eq:Gaussian-reduction-energy}
\mathcal E[\rho]-\mathcal E[G_\rho]
=\gamma\mathcal H(\rho\mid G_\rho)\geq0.
\end{equation}
Here the relative entropy of $\rho$ with respect to $G_\rho$ is defined by
\[
\mathcal H(\rho\mid G_\rho)
:=\int_{\bbr^\dm}
\log\left(\frac{\d\rho}{\d G_\rho}\right)\,\d\rho.
\]
In particular, equality in \eqref{eq:Gaussian-reduction-energy} holds if and only if $\rho=G_\rho$.
Consequently, every global minimizer of $\mathcal E$ is a nondegenerate
Gaussian measure.
\end{proposition}

\begin{proof}
Since $\mathcal E[\rho]<\infty$, the entropy is finite and hence $\rho\ll\mathcal L^\dm$. Let $f:=\d\rho/\d\mathcal L^\dm$ denote its density. If the covariance matrix $\Sigma$ is degenerate, there exists a nonzero vector $\xi$ such that
$\xi^\top\Sigma\xi=0$, which implies
\[
\int_{\bbr^\dm}|\xi\cdot(x-m)|^2\,\d\rho(x)=0.
\]
Hence $\rho$ is concentrated on the hyperplane
$\{x:\xi\cdot(x-m)=0\}$, which has zero Lebesgue measure.
This contradicts $\rho\ll\mathcal L^\dm$ and
$\rho(\bbr^\dm)=1$, which shows that $\Sigma$ must be positive definite.

Now, let $g_\rho:=\d G_\rho/\d\mathcal L^\dm$ denote the density of $G_\rho$. Then
\[
g_\rho(x)
=\frac{1}{(2\pi)^{\dm/2}(\det\Sigma)^{1/2}}
\exp\left(-\frac12(x-m)^\top\Sigma^{-1}(x-m)\right),
\]
and taking logarithms gives
\begin{equation}
\label{eq:log-Gaussian-density}
\log g_\rho(x)
=-\frac{\dm}{2}\log(2\pi)
 -\frac12\log\det\Sigma
 -\frac12(x-m)^\top\Sigma^{-1}(x-m).
\end{equation}
Note that the integral of the quadratic term in \eqref{eq:log-Gaussian-density} is
completely determined by the covariance matrix. Indeed,
\begin{align*}
&\int_{\bbr^\dm}
(x-m)^\top\Sigma^{-1}(x-m)\,\d\rho(x)\\
&\qquad
=\operatorname{tr}\left(
\Sigma^{-1}
\int_{\bbr^\dm}(x-m)(x-m)^\top\,\d\rho(x)
\right)
=\operatorname{tr}(\Sigma^{-1}\Sigma)
=\dm.
\end{align*}
In addition, since $G_\rho$ has the same mean and covariance as $\rho$, the same
calculation gives
\[
\int_{\bbr^\dm}
(x-m)^\top\Sigma^{-1}(x-m)\,\d G_\rho(x)=\dm.
\]
The preceding calculations and \eqref{eq:log-Gaussian-density} give
\begin{equation}
\label{eq:cross-entropy-Gaussian}
\int_{\bbr^\dm}\log g_\rho\,\d\rho
=\int_{\bbr^\dm}\log g_\rho\,\d G_\rho.
\end{equation}
As $\rho=f\mathcal L^\dm$ and
$G_\rho=g_\rho\mathcal L^\dm$ with $g_\rho>0$,
we have $\rho\ll G_\rho$ and
$\d\rho/\d G_\rho=f/g_\rho$. Thus,
\begin{align*}
\mathcal H(\rho\mid G_\rho)
&=\int_{\bbr^\dm}\log\left(\frac{f}{g_\rho}\right)\,\d\rho\\
&=\int_{\bbr^\dm}f\log f\,\d x
  -\int_{\bbr^\dm}\log g_\rho\,\d\rho\\
&=\operatorname{Ent}(\rho)-\operatorname{Ent}(G_\rho).
\end{align*}
The interaction energies also agree, since the covariance formula
from Section~\ref{sec:preliminaries} gives
\[
\mathcal V_n[\rho]
=c_n e_n(\Sigma)
=\mathcal V_n[G_\rho].
\]
It follows that
\[
\mathcal E[\rho]-\mathcal E[G_\rho]
=\gamma\bigl(\operatorname{Ent}(\rho)-\operatorname{Ent}(G_\rho)\bigr)
=\gamma\mathcal H(\rho\mid G_\rho).
\]
Finally, Jensen's inequality applied to the strictly convex
function $r\mapsto r\log r$ yields
$\mathcal H(\rho\mid G_\rho)\geq0$,
with equality if and only if $\rho=G_\rho$, which shows that every global minimizer of $\mathcal{E}$ must be a Gaussian measure.
\end{proof}

By Proposition~\ref{prop:Gaussian-reduction-minimizer},
it suffices to consider nondegenerate Gaussian measures
$G_{m,\Sigma}:=\mathcal N(m,\Sigma)$ for energy minimization. Write
\[
\Sigma=P^\top\operatorname{diag}(\lambda_1,\ldots,\lambda_\dm)P,
\qquad \lambda_i>0,
\]
where $P$ is orthogonal. Then
\begin{equation}
\label{eq:Gaussian-finite-dimensional-energy}
\mathcal E[G_{m,\Sigma}]
=-\frac{\gamma\dm}{2}\log(2\pi e)
-\frac{\gamma}{2}\log\left(\prod_{i=1}^\dm\lambda_i\right)
+c_n e_n(\lambda_1,\ldots,\lambda_\dm).
\end{equation}
Since this expression depends neither on $m$ nor on $P$,
the minimization reduces to the covariance eigenvalues
$\lambda_1,\ldots,\lambda_\dm>0$.

\subsection{The case $n=\dm$}

When $n=\dm$,
\[
 e_\dm(\lambda_1,\ldots,\lambda_\dm)
 =\prod_{i=1}^\dm\lambda_i.
\]
Writing $p:=\det\Sigma=\prod_i\lambda_i$, the covariance-dependent part of
\eqref{eq:Gaussian-finite-dimensional-energy} becomes
\[
 h_\dm(p):=-\frac{\gamma}{2}\log p+c_\dm p,
 \qquad p>0.
\]
This function has the unique minimizer
\[
 p_\ast=\frac{\gamma}{2c_\dm}
 =\frac{\gamma\dm!}{2(\dm+1)}.
\]
Consequently, the global minimizers are precisely the Gaussian
measures $\mathcal N(m,\Sigma)$, where $m\in\bbr^\dm$ is arbitrary
and $\Sigma$ is any positive definite matrix satisfying
\begin{equation}
\label{eq:minimizer-determinant-nd}
\det\Sigma=\frac{\gamma\dm!}{2(\dm+1)}.
\end{equation}
Thus, when $n=\dm$, the energy determines the determinant of
the covariance matrix but not its shape.

\subsection{The case $1\leq n<\dm$}

Assume $1\leq n<\dm$. By Maclaurin's inequality,
\begin{equation}
\label{eq:Maclaurin-minimizers}
e_n(\lambda_1,\ldots,\lambda_\dm)
\geq\binom{\dm}{n}
\left(\prod_{i=1}^\dm\lambda_i\right)^{n/\dm},
\end{equation}
with equality if and only if
$\lambda_1=\cdots=\lambda_\dm$.
Hence, with $p:=\prod_i\lambda_i$,
equation \eqref{eq:Gaussian-finite-dimensional-energy} gives
\[
\mathcal E[G_{m,\Sigma}]
\geq-\frac{\gamma\dm}{2}\log(2\pi e)
-\frac{\gamma}{2}\log p
+c_n\binom{\dm}{n}p^{n/\dm}.
\]
Writing this lower bound as
\[
H(p)
:=-\frac{\gamma\dm}{2}\log(2\pi e)
-\frac{\gamma}{2}\log p
+c_n\binom{\dm}{n}p^{n/\dm},
\qquad p>0,
\]
we have $\mathcal E[G_{m,\Sigma}]\geq H(\det\Sigma)$,
with equality precisely when all covariance eigenvalues
are equal.

Since the function $r\mapsto H(e^r)$ is strictly convex
and coercive on $\mathbb R$, the function $H$ has a
unique minimizer $p_\ast>0$, characterized by
\[
p_\ast^{\,n/\dm}
=\frac{\gamma}
{2c_n\binom{\dm-1}{n-1}}.
\]
At this minimizer, the isotropic covariance
$\Sigma=\lambda_\ast I_\dm$, where
$\lambda_\ast:=p_\ast^{1/\dm}$, satisfies
$\det\Sigma=p_\ast$ and attains equality in
\eqref{eq:Maclaurin-minimizers}.

For every positive definite covariance matrix $\Sigma$, we have
\[
\mathcal E[G_{m,\Sigma}]
\geq H(\det\Sigma)
\geq H(p_\ast)
=\mathcal E[G_{m,\lambda_\ast I_\dm}].
\]
In particular, equality holds only if $\det\Sigma=p_\ast$ and all
covariance eigenvalues are equal. Thus the minimizing
covariance is uniquely determined by
$\Sigma=\lambda_\ast I_\dm$, where
\begin{equation}
\label{eq:isotropic-minimizer-scale}
\lambda_\ast
=
\left(
\frac{\gamma n!}
{2(n+1)\binom{\dm-1}{n-1}}
\right)^{1/n}.
\end{equation}
For each fixed mean $m\in\bbr^\dm$, the unique
minimizer with that mean is therefore
\begin{equation}
\label{eq:isotropic-global-minimizer}
G_{m,\lambda_\ast}
:=\mathcal N(m,\lambda_\ast I_\dm).
\end{equation}

\subsection{A volume-preserving balancing interpretation}
\label{subsec:balancing-minimizers}

The distinction between $n<\dm$ and $n=\dm$ also admits a
geometric interpretation in terms of volume-preserving
deformations. For a Gaussian measure with covariance eigenvalues
$\lambda_1,\ldots,\lambda_\dm$, the entropy depends on the
covariance only through
\[
\det\Sigma=\prod_{k=1}^\dm\lambda_k.
\]
Fix $i\neq j$ and replace the two eigenvalues by their
geometric mean,
\[
\widetilde\lambda_i=\widetilde\lambda_j
=\sqrt{\lambda_i\lambda_j},
\qquad
\widetilde\lambda_k=\lambda_k
\quad\text{for }k\notin\{i,j\}.
\]
This replacement preserves the determinant and hence the
entropy. In an eigenbasis of $\Sigma$, it is induced by a
linear map of determinant one.

Using the notation $e_k(\hat\lambda_i,\hat\lambda_j)$
introduced in Section~\ref{sec:preliminaries}, we have
\[
\begin{aligned}
e_n(\lambda_1,\ldots,\lambda_\dm)
={}&e_n(\hat\lambda_i,\hat\lambda_j)\\
&+(\lambda_i+\lambda_j)
e_{n-1}(\hat\lambda_i,\hat\lambda_j)\\
&+\lambda_i\lambda_j
e_{n-2}(\hat\lambda_i,\hat\lambda_j).
\end{aligned}
\]
Since the product $\lambda_i\lambda_j$ is unchanged,
this identity gives
\begin{equation}
\label{eq:pairwise-balancing-drop}
\begin{aligned}
&e_n(\lambda_1,\ldots,\lambda_\dm)
-e_n(\widetilde\lambda_1,\ldots,\widetilde\lambda_\dm)\\
&\qquad=
\left(\sqrt{\lambda_i}-\sqrt{\lambda_j}\right)^2
e_{n-1}(\hat\lambda_i,\hat\lambda_j).
\end{aligned}
\end{equation}
When $n<\dm$, the polynomial
$e_{n-1}(\hat\lambda_i,\hat\lambda_j)$ is strictly
positive because $n-1\leq\dm-2$ and all eigenvalues
are positive. Thus balancing any unequal pair preserves
the entropy while strictly decreasing the interaction
energy. This provides a geometric explanation for why
the minimizing covariance must be isotropic. When $n=\dm$, however, the interaction depends only on
the determinant of the covariance. It is therefore
unchanged by any determinant-preserving deformation,
just like the entropy. This invariance extends to volume-preserving affine transformations.

To see this, let $\rho\in\mathcal P_2(\bbr^\dm)$, $m=m(\rho)$, and
$B\in\mathrm{GL}(\dm)$ with $|\det B|=1$. For the affine map
$\mathcal A_{m,B}(x):=m+B(x-m)$, we have
\[
\begin{aligned}
\Sigma\bigl((\mathcal A_{m,B})_\#\rho\bigr)
&=B\Sigma(\rho)B^\top,\\
\operatorname{Ent}\bigl((\mathcal A_{m,B})_\#\rho\bigr)
&=\operatorname{Ent}(\rho).
\end{aligned}
\]
Since $\mathcal V_\dm[\rho]=c_\dm\det\Sigma(\rho)$ when
$n=\dm$, the free energy is also invariant:
\[
\mathcal E\bigl[(\mathcal A_{m,B})_\#\rho\bigr]
=\mathcal E[\rho].
\]
In particular, when $n=\dm$, every Gaussian minimizer with a fixed mean can be obtained from the isotropic minimizer by a volume-preserving affine transformation.

We next show that the Gaussian minimizers classified above coincide with the stationary weak solutions.

\begin{proposition}[Minimizers and stationary states]
\label{prop:minimizers-stationary}
A measure $\rho\in\mathcal P_2(\bbr^\dm)$ is a stationary
weak solution of \eqref{PDE} if and only if it is one of
the Gaussian minimizers classified above.
\end{proposition}

\begin{proof}
Let $\Sigma_\ast$ be any minimizing covariance classified
above. Since $\Sigma_\ast$ minimizes the Gaussian free
energy \eqref{eq:Gaussian-finite-dimensional-energy}
over positive definite covariance matrices, its first
variation vanishes in every symmetric direction $H$.
Using the differential formula for $e_n$ established
in Section~\ref{sec:preliminaries}, we obtain
\[
\begin{aligned}
0
&=\left.\frac{\d}{\d\varepsilon}
\mathcal E[\mathcal N(m,\Sigma_\ast+\varepsilon H)]
\right|_{\varepsilon=0}\\
&=\operatorname{tr}\left[
\left(c_nT_{n-1}(\Sigma_\ast)
-\frac{\gamma}{2}\Sigma_\ast^{-1}\right)H
\right].
\end{aligned}
\]
Since this holds for every symmetric matrix $H$, we have
\[
2c_nT_{n-1}(\Sigma_\ast)=\gamma\Sigma_\ast^{-1}.
\]
Now, let $g$ be the density of $\mathcal N(m,\Sigma_\ast)$.
By the drift formula derived in
Section~\ref{sec:preliminaries},
\[
\begin{aligned}
v[\mathcal N(m,\Sigma_\ast)](x)g(x)
&=-2c_nT_{n-1}(\Sigma_\ast)(x-m)g(x)\\
&=-\gamma\Sigma_\ast^{-1}(x-m)g(x)\\
&=\gamma\nabla g(x).
\end{aligned}
\]
Taking the divergence of the above formula gives the stationary equation
\eqref{PDE}, and thus every Gaussian minimizer is a
stationary weak solution.

Conversely, let $\rho$ be a stationary weak solution,
and set $m:=m(\rho)$ and $\Sigma:=\Sigma(\rho)$.
By Proposition~\ref{prop:moment-equations-diffusion}
and \eqref{eq:covariance-diffusion}, stationarity gives
\[
4c_n\Sigma T_{n-1}(\Sigma)=2\gamma I_\dm.
\]
In particular, $\Sigma$ is positive definite.
With $A:=2c_nT_{n-1}(\Sigma)$, the preceding identity
becomes $\Sigma A=\gamma I_\dm$, and hence
$A=\gamma\Sigma^{-1}$ is positive definite.

We first show that $\rho$ is Gaussian. Since $\rho$
is stationary, the constant curve $\rho_t\equiv\rho$
is a weak solution with initial datum $\rho$.
By uniqueness in
Theorem~\ref{thm:global-wellposedness-diffusion},
the representation \eqref{eq:explicit-solution-diffusion}
applies with $\rho_0=\rho$. Thus, for every $t\geq0$,
\[
\rho=
\bigl(m+\Phi(t,0)(\cdot-m)\bigr)_\#\rho
*\mathcal N(0,Q_t).
\]
Since the covariance remains equal to $\Sigma$,
equation \eqref{eq:def-At-diffusion} gives
$A_t\equiv A$. The fundamental matrix equation
\eqref{eq:fundamental-matrix-diffusion} then yields
$\Phi(t,0)=e^{-At}$.

Write
$\Sigma=P^\top\operatorname{diag}
(\lambda_1,\ldots,\lambda_\dm)P$,
where $\lambda_i>0$. Since $A=\gamma\Sigma^{-1}$,
the diagonal representation
\eqref{eq:Phi-diagonal-diffusion} gives
\[
r_i(t,0)=e^{-\gamma t/\lambda_i}.
\]
Moreover, the covariance identity
\eqref{eq:qi-covariance-diffusion} gives
$q_i(t)=\lambda_i(t)-\lambda_i^0r_i(t,0)^2$.
As $\lambda_i(t)=\lambda_i^0=\lambda_i$ by
stationarity, we obtain
\[
q_i(t)=\lambda_i(1-r_i(t,0)^2)
\longrightarrow\lambda_i
\qquad\text{as }t\to\infty.
\]
By \eqref{eq:Phi-diagonal-diffusion} and
\eqref{eq:Q-diagonal-diffusion}, these limits give
$\Phi(t,0)\to0$ and $Q_t\to\Sigma$.
To conclude that $\rho$ is Gaussian, we apply the
Wasserstein upper bound
\eqref{eq:direct-W2-upper-diffusion} with
$\overline\rho=\mathcal N(m,\Sigma)$. As $\rho_t=\rho$, we obtain
\[
\begin{aligned}
W_2^2\bigl(\rho,\mathcal N(m,\Sigma)\bigr)
&\leq\sum_{i=1}^\dm\lambda_i r_i(t,0)^2\\
&\quad+\sum_{i=1}^\dm\lambda_i
\left(\sqrt{1-r_i(t,0)^2}-1\right)^2.
\end{aligned}
\]
Since $r_i(t,0)\to0$ for every $i$, the right-hand
side tends to zero as $t\to\infty$. Therefore,
$\rho=\mathcal N(m,\Sigma)$.

It remains to show that $\Sigma$ is one of the
minimizing covariances classified above.
By the eigenvalue equation
\eqref{eq:eigenvalue-diffusion}, stationarity gives
\[
\lambda_i e_{n-1}(\hat\lambda_i)
=\frac{\gamma}{2c_n},
\qquad 1\leq i\leq\dm.
\]
If $n<\dm$, then
$e_{n-1}(\hat\lambda_i,\hat\lambda_j)>0$
for every $i\neq j$, since all eigenvalues are
positive and $n-1\leq\dm-2$.
Since the eigenvalues are constant and the exponential
factor in \eqref{eq:eigenvalue-difference-diffusion}
is strictly less than one for every $t>0$, we must
have $\lambda_i=\lambda_j$ for all $i,j$.
Their common value $\lambda$ satisfies
\[
\binom{\dm-1}{n-1}\lambda^n
=\frac{\gamma}{2c_n}.
\]
Since $\lambda>0$, this equation uniquely determines
$\lambda$, which coincides with $\lambda_\ast$ given
in \eqref{eq:isotropic-minimizer-scale}.
Thus, we have $\Sigma=\lambda_\ast I_\dm$.

If $n=\dm$, then
$\lambda_i e_{\dm-1}(\hat\lambda_i)=\det\Sigma$,
so the stationary eigenvalue equation reduces to
\[
\det\Sigma=\frac{\gamma}{2c_\dm},
\]
which is precisely
\eqref{eq:minimizer-determinant-nd}.
Thus, in both cases, $\rho$ is one of the
Gaussian minimizers classified above.
\end{proof}

\subsection{Lack of displacement convexity}
\label{subsec:nonconvexity}

We close this section by showing that, for $n\geq2$, the free energy
lies outside the framework of displacement convexity. Recall that a
functional
$\mathcal F:\mathcal P_2(\bbr^\dm)\to(-\infty,+\infty]$ is called
$\lambda$-convex along geodesics, for some $\lambda\in\bbr$, if every
pair of measures $\mu^0,\mu^1$ with finite values of $\mathcal F$ is
connected by a constant-speed geodesic $(\mu^s)_{s\in[0,1]}$ of
$\bigl(\mathcal P_2(\bbr^\dm),W_2\bigr)$ such that
\[
\mathcal F[\mu^s]
\leq(1-s)\mathcal F[\mu^0]+s\mathcal F[\mu^1]
-\frac{\lambda}{2}s(1-s)W_2^2(\mu^0,\mu^1),
\qquad s\in[0,1].
\]
This property underlies the contraction and convergence theory for
Wasserstein gradient flows \cite{AGS2005,carrillo2006contractions}. The
entropy is convex along geodesics, and so is the interaction energy
associated with a convex potential \cite{mccann1997convexity}. In
particular, the free energy is convex along geodesics when $n=1$, since
$\mathcal V_1[\rho]=\iint|x-y|^2\,\d\rho(x)\,\d\rho(y)$. This is no
longer true for higher-dimensional simplices.

\begin{proposition}[Failure of displacement semiconvexity]
\label{prop:nonconvexity}
Let $2\leq n\leq\dm$. Then the free energy $\mathcal E$ is not
$\lambda$-convex along geodesics for any $\lambda\in\bbr$.
\end{proposition}

\begin{proof}
Fix $a,b,c>0$ with $a\neq b$, and consider the centered Gaussian
measures $\mu^0=\mathcal N(0,\Sigma^0)$ and
$\mu^1=\mathcal N(0,\Sigma^1)$, where
\[
\Sigma^0=\operatorname{diag}(a,b,c,\ldots,c),
\qquad
\Sigma^1=\operatorname{diag}(b,a,c,\ldots,c),
\]
and the entry $c$ appears $\dm-2$ times. The linear map
\[
x\longmapsto Lx,
\qquad
L:=\operatorname{diag}\bigl(\sqrt{b/a},\sqrt{a/b},1,\ldots,1\bigr),
\]
is the gradient of a convex function and pushes $\mu^0$ forward to
$\mu^1$. It is therefore the optimal transport map between these
measures. Since $\mu^0$ is absolutely continuous, the unique
constant-speed geodesic connecting $\mu^0$ to $\mu^1$ is given by
$\mu^s=\bigl((1-s)\Id+sL\bigr)_\#\mu^0$
(see, for instance, \cite{AGS2005,villani2009optimal}). Hence
$\mu^s=\mathcal N(0,\Sigma^s)$ with
\[
\begin{aligned}
\Sigma^s&=\operatorname{diag}\bigl(u(s)^2,w(s)^2,c,\ldots,c\bigr),\\
u(s)&:=(1-s)\sqrt a+s\sqrt b,
\qquad
w(s):=(1-s)\sqrt b+s\sqrt a.
\end{aligned}
\]
Setting $\delta:=(\sqrt a-\sqrt b)^2>0$, a direct computation gives
\[
\begin{aligned}
W_2^2(\mu^0,\mu^1)
&=\int_{\bbr^\dm}|x-Lx|^2\,\d\mu^0(x)=2\delta,\\
u(s)w(s)&=\sqrt{ab}+s(1-s)\delta,\\
u(s)^2+w(s)^2&=a+b-2s(1-s)\delta.
\end{aligned}
\]

We now evaluate the free energy along this geodesic. For
$k\in\bbz$, let $\vartheta_k:=\binom{\dm-2}{k}c^k$ denote the
$k$-th elementary symmetric polynomial of $\dm-2$ copies of $c$, with
the convention $\vartheta_k=0$ if $k<0$ or $k>\dm-2$. Separating the
first two eigenvalues as in
Section~\ref{subsec:balancing-minimizers}, we have
\[
e_n(\Sigma^s)
=\vartheta_n
+\bigl(u(s)^2+w(s)^2\bigr)\vartheta_{n-1}
+u(s)^2w(s)^2\vartheta_{n-2}.
\]
By \eqref{eq:Gaussian-finite-dimensional-energy}, the function
$\phi(s):=\mathcal E[\mu^s]$ is therefore given by
\[
\begin{aligned}
\phi(s)
={}&-\frac{\gamma\dm}{2}\log(2\pi e)
-\frac{\gamma(\dm-2)}{2}\log c
-\gamma\log\bigl(u(s)w(s)\bigr)\\
&+c_n\Bigl[
\vartheta_n
+\bigl(u(s)^2+w(s)^2\bigr)\vartheta_{n-1}
+u(s)^2w(s)^2\vartheta_{n-2}
\Bigr].
\end{aligned}
\]
At $s=1/2$, the product $uw$ takes the value $(\sqrt a+\sqrt b)^2/4$,
its first derivative vanishes, and its second derivative equals
$-2\delta$. Differentiating $\phi$ twice, we obtain
\begin{equation}
\label{eq:second-derivative-geodesic}
\phi''(1/2)
=\frac{8\gamma\delta}{(\sqrt a+\sqrt b)^2}
+c_n\delta\Bigl[
4\vartheta_{n-1}-(\sqrt a+\sqrt b)^2\vartheta_{n-2}
\Bigr].
\end{equation}

Suppose now that $\mathcal E$ is $\lambda$-convex along geodesics for
some $\lambda\in\bbr$. For $0\leq s_1<s_2\leq1$, the restriction of
$(\mu^s)$ to $[s_1,s_2]$ is, after an affine reparametrization, the
unique constant-speed geodesic connecting $\mu^{s_1}$ to $\mu^{s_2}$,
and $W_2^2(\mu^{s_1},\mu^{s_2})=2\delta(s_2-s_1)^2$. Applying the
definition of $\lambda$-convexity to these restrictions shows that
the function $s\mapsto\phi(s)-\lambda\delta s^2$ is convex on $[0,1]$.
Consequently, $\phi''(1/2)\geq2\lambda\delta$, and
\eqref{eq:second-derivative-geodesic} yields
\[
\lambda
\leq\frac{4\gamma}{(\sqrt a+\sqrt b)^2}
+c_n\Bigl[
2\vartheta_{n-1}
-\frac12(\sqrt a+\sqrt b)^2\vartheta_{n-2}
\Bigr].
\]
Since $2\leq n\leq\dm$, we have
$\vartheta_{n-2}=\binom{\dm-2}{n-2}c^{n-2}>0$. Choosing $b=4a$ and
letting $a\to\infty$ with $c$ fixed, the right-hand side tends to
$-\infty$, which is a contradiction.
\end{proof}

In the proof of Proposition~\ref{prop:nonconvexity}, the
loss of convexity comes from the term
$u(s)^2w(s)^2\vartheta_{n-2}$, which is absent when $n=1$. The
convergence results of Section~\ref{sec:longtime}
are therefore not consequences of a contraction property. They rely
instead on the explicit covariance dynamics.

\section{Long-time behavior and convergence rates}
\label{sec:longtime}
\setcounter{equation}{0}

For nonnegative functions $f$ and $g$, we write
$f(t)\asymp g(t)$ as $t\to\infty$ if there exist
constants $A,C>0$ and $t_0>0$ such that
\[
Ag(t)\leq f(t)\leq Cg(t),
\qquad t\geq t_0.
\]
Throughout this section, let
$\rho_0\in\mathcal P_2(\bbr^d)$ be arbitrary,
and let $(\rho_t)_{t\geq0}$ be the unique weak
solution of \eqref{PDE} given by
Theorem~\ref{thm:global-wellposedness-diffusion}. For this $\rho_t$, write $\Sigma_t:=\Sigma(\rho_t)$ and choose
an orthogonal matrix $P$ such that
\begin{equation}
\label{eq:initial-covariance-eigenvalues}
\Sigma_0=P^\top\operatorname{diag}
(\lambda_1^0,\ldots,\lambda_d^0)P,
\qquad
\lambda_1^0\geq\cdots\geq\lambda_d^0\geq0.
\end{equation}
By Proposition~\ref{prop:global-reduced-diffusion},
the covariance matrices retain this eigenbasis,
with
\[
\Sigma_t=P^\top\operatorname{diag}
(\lambda_1(t),\ldots,\lambda_d(t))P.
\]
The eigenvalue equation \eqref{eq:eigenvalue-diffusion}
then takes the form
\begin{equation}
\label{eq:eigenvalue-long-time}
\frac{\d}{\d t}\lambda_i(t)
=-4c_n\lambda_i(t)e_{n-1}(\hat\lambda_i(t))
+2\gamma.
\end{equation}
In addition,  \eqref{eq:eigenvalue-difference-diffusion}
shows that the ordering and multiplicities of the
eigenvalues are preserved, while
Proposition~\ref{prop:global-reduced-diffusion}
ensures their strict positivity for $t>0$.

\subsection{Long-time behavior of the covariance matrix}
\label{subsec:covariance-longtime}

We determine the limiting covariance and its
exponential convergence rates, distinguishing
between the full-dimensional case $n=d$ and
the subcritical case $1\leq n<d$.

\subsubsection{The case $n=d$}

When $n=d$, the identity
\[
\lambda_i e_{d-1}(\hat\lambda_i)
=\prod_{j=1}^d\lambda_j=\det\Sigma_t,
\qquad 1\leq i\leq d,
\]
shows that the right-hand side of
\eqref{eq:eigenvalue-long-time} is independent
of the index $i$. The resulting common shift
of the eigenvalues determines both the limiting
covariance and its exact decay rate.

\begin{proposition}[The full-dimensional case]
\label{prop:covariance-nd}
Assume $n=d$, and let
$\lambda_1^0,\ldots,\lambda_d^0$ be as in
\eqref{eq:initial-covariance-eigenvalues}.
There exists a unique number
\[
\overline\mu\in
\left(-\min_{1\leq i\leq d}\lambda_i^0,\infty\right)
\]
such that
\begin{equation}
\label{eq:mu-equilibrium}
\prod_{i=1}^d(\lambda_i^0+\overline\mu)
=\frac{\gamma}{2c_d}
=\frac{\gamma d!}{2(d+1)}.
\end{equation}
The limiting eigenvalues and covariance are given by
\[
\lambda_i^\infty:=\lambda_i^0+\overline\mu,
\qquad
\Sigma_\infty:=\Sigma_0+\overline\mu I_d,
\]
and $\Sigma_t\to\Sigma_\infty$
as $t\to\infty$.
If $\overline\mu\neq 0$, the convergence has the
sharp exponential rate
\begin{equation}
\label{eq:mu-two-sided-rate}
|\lambda_i(t)-\lambda_i^\infty|
\asymp
\exp\left(
-2\gamma\operatorname{tr}(\Sigma_\infty^{-1})t
\right),
\qquad 1\leq i\leq d.
\end{equation}
If $\overline\mu=0$, then
$\Sigma_t=\Sigma_0$ for every $t\geq0$.
\end{proposition}

\begin{proof}
Since the right-hand side of
\eqref{eq:eigenvalue-long-time} is independent
of $i$, the differences satisfy
\[
\frac{\d}{\d t}(\lambda_i-\lambda_j)=0.
\]
There is consequently a scalar function
$\mu(t)$ with $\mu(0)=0$ such that
$\lambda_i(t)=\lambda_i^0+\mu(t)$ for every $i$.
Substitution into \eqref{eq:eigenvalue-long-time} yields the autonomous equation \[
\mu'(t)=f(\mu(t)),
\qquad
f(\mu):=2\gamma
-4c_d\prod_{i=1}^d(\lambda_i^0+\mu).
\]
Note that the function $f$ is strictly decreasing on
$[-\min_i\lambda_i^0,\infty)$ and satisfies
\[
f(-\min_i\lambda_i^0)=2\gamma>0,
\qquad
\lim_{\mu\to\infty}f(\mu)=-\infty.
\]
It therefore has a unique zero $\overline\mu$,
characterized by \eqref{eq:mu-equilibrium}.
Since $\mu(0)=0\geq-\min_i\lambda_i^0$ and
$f$ changes sign from positive to negative at
$\overline\mu$, the solution converges monotonically
to this equilibrium. That is,
\[
\Sigma_t=\Sigma_0+\mu(t)I_d
\longrightarrow
\Sigma_0+\overline\mu I_d=\Sigma_\infty.
\]
When $\overline\mu=0$, uniqueness of the solution gives $\mu(t)=0$ for all $t\geq0$.

Suppose henceforth that $\overline\mu\neq0$. By
\eqref{eq:mu-equilibrium}, differentiation at the equilibrium gives
\[
\begin{aligned}
f'(\overline\mu)=-4c_d\sum_{j=1}^d
\prod_{k\neq j}\lambda_k^\infty=-2\gamma\sum_{j=1}^d
\frac{1}{\lambda_j^\infty}=-2\gamma\operatorname{tr}
(\Sigma_\infty^{-1})<0,
\end{aligned}
\]
and the uniqueness prevents $\mu(t)$ from reaching
$\overline\mu$ at any finite time. Since
$f(\overline\mu)=0$ and $f\in C^2$,
Taylor's theorem yields
\[
\frac{\d}{\d t}\log|\mu(t)-\overline\mu|
=\frac{f(\mu(t))}{\mu(t)-\overline\mu}
=f'(\overline\mu)
+O(|\mu(t)-\overline\mu|).
\]
Then, the convergence of $\mu(t)$ and the strict
negativity of $f'(\overline\mu)$ first give
$|\mu(t)-\overline\mu|\leq Ce^{-ct}$ for some $C,c>0$. The Taylor remainder is consequently
integrable over $[0,\infty)$, and integration
of the logarithmic derivative yields
\[
\log|\mu(t)-\overline\mu|
=f'(\overline\mu)t+O(1).
\]
Finally, using $\lambda_i(t)-\lambda_i^\infty
=\mu(t)-\overline\mu$, we conclude that
\[|\lambda_i(t)-\lambda_i^\infty|
\asymp
\exp\left(
-2\gamma\operatorname{tr}(\Sigma_\infty^{-1})t
\right),
\]
which proves \eqref{eq:mu-two-sided-rate}.
\end{proof}

The result of Proposition \ref{prop:covariance-nd} can be summarized as follows. Since the eigenvalue differences remain constant, an initially anisotropic covariance
converges to an anisotropic equilibrium. Although
the balance between diffusion and interaction
fixes the limiting determinant, it does not
enforce isotropy.
In this sense, the equilibrium is selected by the initial
covariance through its eigenvectors and eigenvalue gaps,
which are conserved by the flow, while the scalar
$\overline\mu$ adjusts the determinant.
This contrasts with the
subcritical case, where the eigenvalues synchronize.

\subsubsection{The case $1\leq n<d$}

In the subcritical case, the eigenvalue differences
decay.
Still, uniform upper and positive lower bounds on
the covariance eigenvalues allow us to turn the
eigenvalue-difference identity into an exponential
synchronization estimate.

\begin{lemma}[Uniform bounds and synchronization]
\label{lem:uniform-bounds-synchronization}
Assume $1\leq n<d$. For each $t_0>0$, there
exist constants $0<L\leq U<\infty$ such that
\[
L\leq\lambda_i(t)\leq U,
\qquad t\geq t_0,\quad 1\leq i\leq d.
\]
Moreover, there exist constants $C,c>0$ such that
\begin{equation}
\label{eq:synchronization-upper}
\max_{1\leq i,j\leq d}
|\lambda_i(t)-\lambda_j(t)|
\leq Ce^{-ct},
\qquad t\geq0.
\end{equation}
\end{lemma}

\begin{proof}
By Proposition~\ref{prop:global-reduced-diffusion},
the eigenvalues are positive for $t>0$ and
remain ordered as
$\lambda_1(t)\geq\cdots\geq\lambda_d(t)$.
Then the inequality
$e_{n-1}(\hat\lambda_d(t))
\geq\binom{d-1}{n-1}\lambda_d(t)^{n-1}$
gives
\[
\lambda_d'(t)
\leq
-4c_n\binom{d-1}{n-1}\lambda_d(t)^n
+2\gamma,
\]
and comparison with the corresponding scalar
equation bounds $\lambda_d(t)$ above on
$[t_0,\infty)$. Moreover,  the nonnegative
differences $\lambda_i(t)-\lambda_d(t)$  are nonincreasing by
\eqref{eq:eigenvalue-difference-diffusion},
 and hence all eigenvalues admit a common
upper bound $U$ on this interval.

 The upper bound now allows us to estimate
the interaction term in the equation for
$\lambda_d(t)$ as
\[
\lambda_d'(t)
\geq
-4c_n\binom{d-1}{n-1}U^{n-1}\lambda_d(t)
+2\gamma.
\]
Since $\lambda_d(t_0)>0$, comparison with the
corresponding linear equation yields
$\lambda_d(t)\geq L>0$ for all $t\geq t_0$.
The ordering also extends this lower bound  to
all covariance eigenvalues.

Since $n<d$, the uniform  positivity gives
\[
e_{n-1}(\hat\lambda_i(t),\hat\lambda_j(t))
\geq
\binom{d-2}{n-1}L^{n-1}>0
\]
for all $i\neq j$ and $t\geq t_0$.
 Integrating \eqref{eq:eigenvalue-difference-diffusion}
therefore  yields
\[
\begin{aligned}
|\lambda_i(t)-\lambda_j(t)|
&\leq|\lambda_i(t_0)-\lambda_j(t_0)|\\
&\quad\times
\exp\left(
-4c_n\binom{d-2}{n-1}L^{n-1}(t-t_0)
\right),
\end{aligned}
\]
and enlarging the multiplicative constant to
cover $[0,t_0]$ proves
\eqref{eq:synchronization-upper}.
\end{proof}

The synchronization estimate also determines
the limiting covariance, since the evolution
of the average of the eigenvalues approaches
a scalar equation with an attracting equilibrium.

\begin{proposition}[The subcritical case]
\label{prop:covariance-nlessd}
Assume $1\leq n<d$, and let $\lambda_\ast$
be the minimizing covariance scale defined in
\eqref{eq:isotropic-minimizer-scale}.
Then $\Sigma_t\to\lambda_\ast I_d$ as
$t\to\infty$. Moreover, there exist constants
$C,c>0$ such that
\begin{equation}
\label{eq:covariance-upper-nlessd}
\max_{1\leq i\leq d}
|\lambda_i(t)-\lambda_\ast|
\leq Ce^{-ct},
\qquad t\geq0.
\end{equation}
\end{proposition}

\begin{proof}
To  describe the evolution of the  average
of the  covariance eigenvalues  and their
spread, we introduce
\[
\overline\lambda(t):=\frac1d\sum_{i=1}^d\lambda_i(t),
\qquad
D(t):=\max_{1\leq i\leq d}
|\lambda_i(t)-\overline\lambda(t)|.
\]
For any fixed $t_0>0$,
Lemma~\ref{lem:uniform-bounds-synchronization}
provides constants $0<L\leq U<\infty$ such that
$L\leq\lambda_i(t)\leq U$ for $t\geq t_0$,
while $D(t)\leq Ce^{-ct}$ for some $C,c>0$.
Averaging \eqref{eq:eigenvalue-long-time}
and using the  combinatorial identity
$\sum_{i=1}^d\lambda_i e_{n-1}(\hat\lambda_i)
=ne_n(\lambda_1,\ldots,\lambda_d)$,
we obtain
\begin{equation}
\label{eq:average-eigenvalue}
\overline\lambda'(t)
=2\gamma-\frac{4c_n n}{d}
e_n(\lambda_1(t),\ldots,\lambda_d(t)).
\end{equation}

To estimate the interaction term in this equation,
write $\delta_i(t):=\lambda_i(t)-\overline\lambda(t)$,
 which satisfies $\sum_{i=1}^d\delta_i(t)=0$.
At the isotropic point
$(\overline\lambda,\ldots,\overline\lambda)$,
all first derivatives of $e_n$ have the same value
$\binom{d-1}{n-1}\overline\lambda^{n-1}$.
 The linear term in the Taylor expansion
therefore vanishes, while the uniform bounds
on the eigenvalues control the second-order
remainder and give
\begin{equation}
\label{eq:en-quadratic-remainder}
\left|
e_n(\lambda_1(t),\ldots,\lambda_d(t))
-\binom dn\overline\lambda(t)^n
\right|
\leq CD(t)^2,
\qquad t\geq t_0.
\end{equation}

Now, consider the scalar function
$F(s):=2\gamma-4c_n\binom{d-1}{n-1}s^n$,
whose unique positive zero is $\lambda_\ast$
by \eqref{eq:isotropic-minimizer-scale}.
Substituting the Taylor estimate
\eqref{eq:en-quadratic-remainder} into the
averaged equation \eqref{eq:average-eigenvalue}
gives
\[
\overline\lambda'(t)
=F(\overline\lambda(t))+R(t),
\qquad
|R(t)|\leq CD(t)^2\leq Ce^{-2ct}.
\]
Since the function $F$ is strictly decreasing on
$(0,\infty)$,  the uniform bounds on
$\overline\lambda(t)$  ensure that $F'$ is
bounded above  by a negative constant on a
fixed interval containing both $\overline\lambda(t)$ and  $\lambda_\ast$.

Writing $y(t):=\overline\lambda(t)-\lambda_\ast$
 and applying the mean value theorem to the
scalar equation, we obtain
\[
\begin{aligned}
\frac{\d}{\d t}|y(t)|
&=\operatorname{sgn}(y(t))
\bigl(F(\overline\lambda(t))-F(\lambda_\ast)+R(t)\bigr)\\
&\leq-a|y(t)|+Ce^{-2ct}
\end{aligned}
\]
for some $a>0$ and almost every $t\geq t_0$.
Since  $y$ is continuously differentiable,
$|y|$ is locally Lipschitz and the displayed
differential inequality holds almost everywhere,
including on the zero set of $y$.
Then, integration of the above inequality gives
\[
|y(t)|
\leq e^{-a(t-t_0)}|y(t_0)|
+C\int_{t_0}^t e^{-a(t-s)}e^{-2cs}\,\d s,
\]
where the right-hand side decays exponentially,
possibly with a smaller positive decay rate
when the two exponents coincide. Together
with the exponential estimate for $D(t)$,
 we obtain $|\lambda_i(t)-\lambda_\ast|
\leq D(t)+|y(t)|\leq Ce^{-ct}$
 for $t\geq t_0$, after adjusting $C,c>0$.
Increasing $C$ to cover $[0,t_0]$  proves
\eqref{eq:covariance-upper-nlessd}.
\end{proof}

 Although Proposition~\ref{prop:covariance-nlessd}
establishes exponential convergence to the
isotropic covariance, it does not determine
the sharp decay rate.  Combining the
eigenvalue-difference equation with the
invariance of the isotropic covariance family
 allows us to identify this rate and its
dependence on the initial covariance.

\begin{corollary}[Sharp covariance decay rates]
\label{cor:sharp-covariance-rates}
Assume $1\leq n<d$, and let $\lambda_\ast$
be as in \eqref{eq:isotropic-minimizer-scale}.
The following assertions hold.
\begin{enumerate}
\item For every pair $i\neq j$ with
$\lambda_i^0\neq\lambda_j^0$, we have
\begin{equation}
\label{eq:sharp-pairwise-rate}
|\lambda_i(t)-\lambda_j(t)|
\asymp
\exp\left(
-\frac{2\gamma(d-n)}{(d-1)\lambda_\ast}t
\right).
\end{equation}

\item If $\Sigma_0$ is not a scalar multiple
of $I_d$, then
\begin{equation}
\label{eq:sharp-anisotropic-covariance-rate}
\max_{1\leq i\leq d}
|\lambda_i(t)-\lambda_\ast|
\asymp
\exp\left(
-\frac{2\gamma(d-n)}{(d-1)\lambda_\ast}t
\right).
\end{equation}

\item If $\Sigma_0=s_0I_d$ for some
$s_0\geq0$ with $s_0\neq\lambda_\ast$, then
\begin{equation}
\label{eq:sharp-isotropic-covariance-rate}
\max_{1\leq i\leq d}
|\lambda_i(t)-\lambda_\ast|
\asymp
\exp\left(
-\frac{2n\gamma}{\lambda_\ast}t
\right).
\end{equation}
If $\Sigma_0=\lambda_\ast I_d$, the covariance
remains stationary for all $t\geq0$.
\end{enumerate}
\end{corollary}

\begin{proof}
\noindent
({1})
Proposition~\ref{prop:covariance-nlessd}
ensures exponential convergence of every
covariance eigenvalue to $\lambda_\ast$.
Since $e_{n-1}$ is a polynomial and the
eigenvalues remain uniformly bounded, we have
\[
e_{n-1}(\hat\lambda_i(t),\hat\lambda_j(t))
=
\binom{d-2}{n-1}\lambda_\ast^{n-1}
+O(e^{-ct})
\]
for some $c>0$. Since the exponentially decaying
error is integrable over $[0,\infty)$, we have
\[
\int_0^t
e_{n-1}(\hat\lambda_i(s),\hat\lambda_j(s))
\,\d s
=
\binom{d-2}{n-1}\lambda_\ast^{n-1}t+O(1).
\]
 Using the exact eigenvalue-difference identity
\eqref{eq:eigenvalue-difference-diffusion},
 we consequently obtain, for $\lambda_i^0\neq\lambda_j^0$,
\[
\begin{aligned}
|\lambda_i(t)-\lambda_j(t)|
&=|\lambda_i^0-\lambda_j^0|
\exp\left(
-4c_n\int_0^t
e_{n-1}(\hat\lambda_i(s),\hat\lambda_j(s))
\,\d s
\right)\\
&\asymp
\exp\left(
-4c_n\binom{d-2}{n-1}
\lambda_\ast^{n-1}t
\right).
\end{aligned}
\]
Then, the equilibrium relation
\eqref{eq:isotropic-minimizer-scale}  and the
binomial identity $\binom{d-2}{n-1}/\binom{d-1}{n-1}
=(d-n)/(d-1)$  give
\[
4c_n\binom{d-2}{n-1}\lambda_\ast^{n-1}
=\frac{2\gamma(d-n)}{(d-1)\lambda_\ast},
\]
which proves \eqref{eq:sharp-pairwise-rate}.

\medskip
\noindent
(2) Suppose that $\Sigma_0$ is not isotropic,
and denote the decay rate in
\eqref{eq:sharp-pairwise-rate} by
\[r:=2\gamma(d-n)/((d-1)\lambda_\ast).\]
Retaining $\overline\lambda(t)$ and $D(t)$
from Proposition~\ref{prop:covariance-nlessd},
we apply the pairwise estimate to obtain
\[
\begin{aligned}
D(t)
&=\max_i\left|
\frac1d\sum_{j=1}^d
(\lambda_i(t)-\lambda_j(t))
\right|\\
&\leq\max_i\frac1d\sum_{j=1}^d
|\lambda_i(t)-\lambda_j(t)|
\leq Ce^{-rt}.
\end{aligned}
\]
Since at least two initial eigenvalues are
distinct, we may select $i\neq j$ with
$\lambda_i^0\neq\lambda_j^0$. The same sharp
pairwise estimate then gives
\[
\max_k|\lambda_k(t)-\lambda_\ast|
\geq\frac12|\lambda_i(t)-\lambda_j(t)|
\geq Ae^{-rt}
\]
for some $A>0$, establishing the lower bound. The upper bound reduces to estimating the
difference between $\overline\lambda(t)$
and $\lambda_\ast$, since
\[|\lambda_i(t)-\lambda_\ast|
\leq D(t)+|\overline\lambda(t)-\lambda_\ast|.\]
Writing $y(t):=\overline\lambda(t)-\lambda_\ast$,
we use \eqref{eq:en-quadratic-remainder} and
the improved bound $D(t)\leq Ce^{-rt}$ to obtain
\[
e_n(\lambda_1(t),\ldots,\lambda_d(t))
=\binom dn(\lambda_\ast+y(t))^n
+O(e^{-2rt}).
\]
Substitution into the averaged equation
\eqref{eq:average-eigenvalue}, followed by
Taylor expansion around $\lambda_\ast$, yields
\[
\begin{aligned}
y'(t)
&=2\gamma-\frac{4c_n n}{d}
e_n(\lambda_1(t),\ldots,\lambda_d(t))\\
&=2\gamma-4c_n\binom{d-1}{n-1}
(\lambda_\ast+y(t))^n+O(e^{-2rt})\\
&=F(\lambda_\ast+y(t))-F(\lambda_\ast)
+O(e^{-2rt})\\
&=-\frac{2n\gamma}{\lambda_\ast}y(t)
+O(|y(t)|^2)+O(e^{-2rt}),
\end{aligned}
\]
where $F$ is the scalar function from the proof
of Proposition~\ref{prop:covariance-nlessd}.

Since the eigenvalues solve the continuously
differentiable system
\eqref{eq:eigenvalue-long-time},
we have $y\in C^1$. Both $|y|$ and
$z(t):=e^{rt}|y(t)|$ are therefore locally
Lipschitz and differentiable almost everywhere.
At times where $y(t)\neq0$, the preceding
expansion gives
\[
\begin{aligned}
\frac{\d}{\d t}|y(t)|
&=\operatorname{sgn}(y(t))y'(t)\\
&\leq-\frac{2n\gamma}{\lambda_\ast}|y(t)|
+C|y(t)|^2+Ce^{-2rt}.
\end{aligned}
\]
The same inequality holds almost everywhere
on the zero set of $y$ where the derivative
of $|y|$ vanishes almost everywhere.
Differentiating $z$ and using $r\leq2n\gamma/\lambda_\ast$, we consequently
obtain
\[
\begin{aligned}
z'(t)
&=re^{rt}|y(t)|
+e^{rt}\frac{\d}{\d t}|y(t)|\\
&\leq\left(r-\frac{2n\gamma}{\lambda_\ast}\right)z(t)
+C|y(t)|z(t)+Ce^{-rt}\\
&\leq C|y(t)|z(t)+Ce^{-rt}
\end{aligned}
\]
for almost every sufficiently large $t$.

Recall that Proposition~\ref{prop:covariance-nlessd} already gives exponential convergence of $y(t)$ to zero (without any explicit estimate of decay rate),
which implies $y\in L^1(0,\infty)$. Therefore, Gronwall's inequality applied from a sufficiently
large time $T$ gives
\[
\begin{aligned}
z(t)
&\leq
\left(z(T)+C\int_T^t e^{-rs}\,\d s\right)
\exp\left(C\int_T^t|y(s)|\,\d s\right)\\
&\leq
\left(z(T)+C\int_T^\infty e^{-rs}\,\d s\right)
\exp\left(C\int_T^\infty|y(s)|\,\d s\right)
<\infty.
\end{aligned}
\]
Finally, the resulting estimate $|y(t)|\leq Ce^{-rt}$,
together with $D(t)\leq Ce^{-rt}$, yields
\[
\max_i|\lambda_i(t)-\lambda_\ast|
\leq D(t)+|y(t)|\leq Ce^{-rt},
\]
which establishes the upper estimate in \eqref{eq:sharp-anisotropic-covariance-rate}.

\medskip
\noindent
(3) For isotropic initial  covariance
$\Sigma_0=s_0I_d$, uniqueness of the
eigenvalue system preserves isotropy, and
every eigenvalue equals $s(t)$, where
$s'(t)=F(s(t))$ and $s(0)=s_0$.
If $s_0\neq\lambda_\ast$, uniqueness prevents
$s(t)$ from reaching $\lambda_\ast$ in
finite time, and Taylor's theorem yields
\[
\frac{\d}{\d t}\log|s(t)-\lambda_\ast|
=\frac{F(s(t))-F(\lambda_\ast)}
{s(t)-\lambda_\ast}
=F'(\lambda_\ast)+O(|s(t)-\lambda_\ast|),
\]
where the exponential convergence established in
Proposition~\ref{prop:covariance-nlessd}
makes the remainder integrable. Integrating
this relation gives
\[
\log|s(t)-\lambda_\ast|
=-\frac{2n\gamma}{\lambda_\ast}t+O(1),
\]
which proves \eqref{eq:sharp-isotropic-covariance-rate}.
When $s_0=\lambda_\ast$, uniqueness of the
scalar equation obviously gives the stationary solution.
\end{proof}

\begin{remark}[Geometric meaning of the covariance rates]
The distinction between the two covariance
decay rates can be understood from the
linearized covariance dynamics at
$\lambda_\ast I_d$. Every symmetric
perturbation $H$ admits the orthogonal
decomposition
\[
H=\frac{\operatorname{tr}H}{d}I_d
+\left(H-\frac{\operatorname{tr}H}{d}I_d\right)
\]
into an isotropic component and a traceless
component with respect to the Frobenius
inner product. The former represents a common
change in the covariance eigenvalues, while
the latter describes their deviations from
the average. These components evolve
independently under the linearized dynamics,
with decay rates determined by the scalar
equation $s'=F(s)$ and the eigenvalue-difference
identity \eqref{eq:eigenvalue-difference-diffusion},
respectively. The corresponding rates are
\[
\frac{2n\gamma}{\lambda_\ast}
\quad\text{and}\quad
\frac{2\gamma(d-n)}{(d-1)\lambda_\ast}.
\]
Since the isotropic covariance family is
invariant under the full nonlinear dynamics,
an initially isotropic covariance remains
on this family and never develops a traceless
component. Its convergence is therefore
governed entirely by the isotropic rate,
which is strictly larger than the anisotropic
rate when $n>1$. The two rates coincide
when $n=1$.
\end{remark}

\subsection{Convergence of the full measure-valued solution}

The preceding analysis determines the limiting covariance matrix and its exponential convergence rates. However, convergence of the covariance alone
does not imply convergence of the full measure-valued solution. We now combine these covariance estimates with the explicit representation
\eqref{eq:explicit-solution-diffusion} to examine the decay of the contribution from the initial distribution and establish convergence in $W_2$.

To identify the limiting measure, we retain the common eigenbasis $P$ fixed in
\eqref{eq:initial-covariance-eigenvalues} and set $\lambda_i^\infty=\lambda_i^0+\overline\mu$
when $n=d$, as determined in
Proposition~\ref{prop:covariance-nd}, and
$\lambda_i^\infty=\lambda_\ast$ when
$1\leq n<d$. Since all limiting eigenvalues
are strictly positive and the center of mass
$m_0:=m(\rho_0)$ is conserved, the corresponding
Gaussian measure is
\[
\Sigma_\infty
=P^\top\operatorname{diag}
(\lambda_1^\infty,\ldots,\lambda_d^\infty)P,
\qquad
\rho_\infty:=\mathcal N(m_0,\Sigma_\infty).
\]
The limiting covariance is isotropic in the
subcritical case, whereas it may retain the
anisotropy of the initial covariance when $n=d$.

The explicit representation separates the transported
initial distribution from an independent Gaussian
component. While a direct coupling controls the
transported contribution by the size of the
transport matrix, the Gaussian convolution permits
a stronger estimate because the centered initial
distribution has zero mean. We use the following
Gaussian smoothing lemma, whose proof is provided
in Appendix~\ref{app:gaussian-smoothing}.

\begin{lemma}[Gaussian smoothing estimate]
\label{lem:gaussian-smoothing}
Let $X$ be a centered random vector in
$\mathbb R^d$ with finite second moment, and
write $\Sigma_X:=\mathbb E[XX^\top]$.
Fix constants $0<q_-<q_+<\infty$.
Let $R$ be a deterministic $d\times d$ matrix
with $\varepsilon:=\|R\|_{\mathrm{op}}\leq1$,
and let $Z$ be independent of $X$ with
Gaussian distribution $\mathcal N(0,Q)$,
where
\[
q_-I_d\leq Q\leq q_+I_d.
\]
Then
\[
W_2\left(
\operatorname{Law}(RX+Z),
\mathcal N(0,Q+R\Sigma_XR^\top)
\right)
=o(\varepsilon)
\qquad\text{as }\varepsilon\to0,
\]
uniformly over the matrices $R$ and $Q$
satisfying the stated conditions.

If, in addition, $\mathbb E|X|^4<\infty$,
there exists a constant $C>0$, depending
only on the moment bounds and $q_-,q_+$,
such that
\begin{equation}
\label{eq:gaussian-smoothing-W2}
W_2\left(
\operatorname{Law}(RX+Z),
\mathcal N(0,Q+R\Sigma_XR^\top)
\right)
\leq C\varepsilon^2.
\end{equation}
\end{lemma}

To apply this lemma to our equation, let
$X=X_0-m_0$, where $X_0$ has distribution
$\rho_0$, and take $R=\Phi(t,0)$ and
$Q=Q_t$ from
\eqref{eq:explicit-solution-diffusion}.
The covariance identity
\eqref{eq:covariance-representation-diffusion} then
gives
\[
Q_t+\Phi(t,0)\Sigma_0\Phi(t,0)^\top
=\Sigma_t.
\]
Thus the lemma controls the distance between
$\rho_t$ and the Gaussian measure
$\mathcal N(m_0,\Sigma_t)$, which has exactly
the same mean and covariance as $\rho_t$.
The positivity of the limiting covariance
ensures that the Gaussian covariance $Q_t$
satisfies the required uniform bounds for
all sufficiently large times, as verified
below.

The remaining difference between
$\mathcal N(m_0,\Sigma_t)$ and $\rho_\infty$
is governed by the covariance convergence
rates established in Section~\ref{subsec:covariance-longtime}.
Combining these two estimates yields
Wasserstein convergence for arbitrary initial
measures in $\mathcal P_2(\mathbb R^d)$
and improved exponential rates when the
initial measure has a finite fourth moment.

\begin{theorem}[Convergence to equilibrium and exponential rates]
\label{thm:explicit-W2-decay}
Let $\rho_0\in\mathcal P_2(\bbr^d)$, and let
$(\rho_t)_{t\geq0}$ be the unique weak solution of
\eqref{PDE}. Denote the limiting Gaussian measure by
$\rho_\infty=\mathcal N(m_0,\Sigma_\infty)$, with covariance
eigenvalues $\lambda_1^\infty,\ldots,\lambda_d^\infty$.
The limiting drift coefficients and the slowest decay rate
of the linear transport factors are
\begin{equation}
\label{eq:limiting-drift-rates}
a_i^\infty:=2c_n e_{n-1}(\hat\lambda_i^\infty)
=\frac{\gamma}{\lambda_i^\infty},\qquad
\beta_*:=\min_{1\leq i\leq d}a_i^\infty
=\frac{\gamma}{\max_i\lambda_i^\infty}.
\end{equation}
The limiting equilibrium and the decay of the full
measure-valued solution satisfy the following assertions.
\begin{enumerate}
\item The limiting Gaussian $\rho_\infty$ is an equilibrium
of \eqref{PDE} and a global minimizer of $\mathcal E$.
Moreover, $\rho_t$ converges to $\rho_\infty$ in $W_2$
as $t\to\infty$.

\item For arbitrary initial data in $\mathcal P_2(\bbr^d)$,
the convergence is exponential. If $\Sigma_0=\Sigma_\infty$,
there exists $C>0$ such that
\begin{equation}
\label{eq:rho-stationary-covariance-upper}
W_2(\rho_t,\rho_\infty)\leq Ce^{-\beta_*t},
\qquad t\geq0.
\end{equation}
When $\Sigma_0\neq\Sigma_\infty$, the covariance approaches
its limit with the sharp exponent $\kappa$ determined by
\begin{equation}
\label{eq:covariance-decay-rate}
\kappa:=
\begin{cases}
2\gamma\operatorname{tr}(\Sigma_\infty^{-1}),
& n=d,\\[1mm]
\displaystyle\frac{2\gamma(d-n)}{(d-1)\lambda_\ast},
& 1\leq n<d\text{ and }\Sigma_0\text{ is not isotropic},\\[3mm]
\displaystyle\frac{2n\gamma}{\lambda_\ast},
& 1\leq n<d\text{ and }\Sigma_0=s_0I_d,
\ s_0\neq\lambda_\ast.
\end{cases}
\end{equation}
In this case, covariance relaxation provides a lower
bound for the Wasserstein distance, while the Gaussian
and transported components together determine an upper
bound. More precisely, there exist $A,C>0$ such that
\begin{equation}
\label{eq:rho-explicit-two-sided}
Ae^{-\kappa t}\leq W_2(\rho_t,\rho_\infty)
\leq C\bigl(e^{-\beta_*t}+e^{-\kappa t}\bigr),
\qquad t\geq0.
\end{equation}
The covariance relaxation therefore determines the sharp
Wasserstein decay rate whenever $\kappa\leq\beta_*$, since
the two bounds then have the same exponential order,
\begin{equation}
\label{eq:rho-sharp-rate}
W_2(\rho_t,\rho_\infty)\asymp e^{-\kappa t}.
\end{equation}

\item If, in addition, $\rho_0\in\mathcal P_4(\bbr^d)$,
the Gaussian smoothing estimate improves the upper bounds
in part~(2). When $\Sigma_0=\Sigma_\infty$, it gives
\begin{equation}
\label{eq:rho-P4-stationary-upper}
W_2(\rho_t,\rho_\infty)\leq Ce^{-2\beta_*t},
\qquad t\geq0.
\end{equation}
If $\Sigma_0\neq\Sigma_\infty$, the contribution from the
transported initial distribution decays at the improved
rate, while covariance relaxation contributes the term
$e^{-\kappa t}$. More precisely,
\begin{equation}
\label{eq:rho-P4-improved-upper}
W_2(\rho_t,\rho_\infty)
\leq C\bigl(e^{-2\beta_*t}+e^{-\kappa t}\bigr),
\qquad t\geq0.
\end{equation}
When $\kappa\leq2\beta_*$, the improved upper bound
and the covariance lower bound in
\eqref{eq:rho-explicit-two-sided} have the same
exponential order, giving the sharp rate
\begin{equation}
\label{eq:rho-P4-sharp-rate}
W_2(\rho_t,\rho_\infty)\asymp e^{-\kappa t}.
\end{equation}
In particular, this condition holds for every $1\leq n<d$
with non-isotropic initial covariance, so the corresponding $W_2$ decay rate
is sharp under the fourth-moment assumption.
\end{enumerate}
\end{theorem}

\begin{proof}
By Proposition~\ref{prop:minimizers-stationary}, the limiting
Gaussian measure $\rho_\infty$ is a global minimizer and
an equilibrium, as asserted in part~(1). We next establish
the exponential bounds, which will also imply the convergence
in that part. The stationary covariance equation gives
$4c_n\lambda_i^\infty e_{n-1}(\hat\lambda_i^\infty)=2\gamma$,
so that $a_i^\infty=\gamma/\lambda_i^\infty$ in
\eqref{eq:limiting-drift-rates}. Write
$a_i(t):=2c_n e_{n-1}(\hat\lambda_i(t))$ for the
time-dependent drift coefficients.

Suppose first that $\Sigma_0\neq\Sigma_\infty$.
Proposition~\ref{prop:covariance-nd} and
Corollary~\ref{cor:sharp-covariance-rates} give the sharp
covariance asymptotic
\begin{equation}
\label{eq:covariance-exact-rate-for-W2}
\max_i|\lambda_i(t)-\lambda_i^\infty|
\asymp e^{-\kappa t}.
\end{equation}
Since $a_i(t)$ depend polynomially on the covariance
eigenvalues, which remain bounded, this implies
$|a_i(t)-a_i^\infty|\leq Ce^{-\kappa t}$. When
$\Sigma_0=\Sigma_\infty$, uniqueness of the covariance
system instead shows that $a_i(t)=a_i^\infty$ for every
time. Thus $a_i-a_i^\infty$ is integrable on $(0,\infty)$
in either case, and
$\int_0^t a_i(s)\,\d s=a_i^\infty t+O(1)$.
The diagonal representation
\eqref{eq:Phi-diagonal-diffusion} consequently yields
\begin{equation}
\label{eq:ri-explicit-asymptotic}
r_i(t,0)=\exp\left(-\int_0^t a_i(s)\,\d s\right)
\asymp e^{-a_i^\infty t},\qquad
r_i(t,0)\leq Ce^{-\beta_*t}.
\end{equation}
The upper bound holds for every $t\geq0$, after an
adjustment of $C$, and controls the part of the solution
transported from the initial distribution in both moment
regimes.

We first prove the estimates in part~(2) for a stationary
covariance, $\Sigma_0=\Sigma_\infty$. In this case
$r_i(t,0)=e^{-a_i^\infty t}$, and the covariance identity
\eqref{eq:qi-covariance-diffusion} reduces to
$q_i(t)=\lambda_i^\infty(1-r_i(t,0)^2)$. The elementary
inequality $1-\sqrt{1-u}\leq u$, valid for $0\leq u\leq1$,
therefore implies
\[|\sqrt{q_i(t)}-\sqrt{\lambda_i^\infty}|^2
\leq\lambda_i^\infty r_i(t,0)^4.\]
Inserting this inequality into the direct coupling bound
\eqref{eq:direct-W2-upper-diffusion} gives
\[
\begin{aligned}
W_2^2(\rho_t,\rho_\infty)
&\leq\sum_{i=1}^d\lambda_i^\infty r_i(t,0)^2
 +\sum_{i=1}^d\lambda_i^\infty r_i(t,0)^4\\
&\leq Ce^{-2\beta_*t}.
\end{aligned}
\]
Taking square roots establishes
\eqref{eq:rho-stationary-covariance-upper} without any
assumption beyond the finite second moment.

For a nonstationary covariance $\Sigma_0\neq\Sigma_\infty$,
the noise covariance
need not be stationary. Its eigenvalues satisfy
$q_i(t)=\lambda_i(t)-\lambda_i^0r_i(t,0)^2$ by
\eqref{eq:qi-covariance-diffusion}.
Combining this identity with the sharp covariance
decay estimate \eqref{eq:covariance-exact-rate-for-W2}
and the transport estimate \eqref{eq:ri-explicit-asymptotic},
we obtain
\begin{equation}
\label{eq:qi-explicit-upper}
|q_i(t)-\lambda_i^\infty|
\leq C\bigl(e^{-\kappa t}+e^{-2\beta_*t}\bigr).
\end{equation}
Since $q_i(t)\to\lambda_i^\infty>0$, the square-root
identity
\[|\sqrt{q_i(t)}-\sqrt{\lambda_i^\infty}|
=|q_i(t)-\lambda_i^\infty|/
(\sqrt{q_i(t)}+\sqrt{\lambda_i^\infty})\]
shows that the difference of the square roots is at most
$C|q_i(t)-\lambda_i^\infty|$ for sufficiently large $t$.
Applying \eqref{eq:direct-W2-upper-diffusion} with
$\overline\Sigma=\Sigma_\infty$, we conclude that
\[
\begin{aligned}
W_2^2(\rho_t,\rho_\infty)
&\leq\sum_{i=1}^d\lambda_i^0 r_i(t,0)^2
+\sum_{i=1}^d
 \bigl(\sqrt{q_i(t)}-\sqrt{\lambda_i^\infty}\bigr)^2\\
&\leq C\bigl(e^{-2\beta_*t}+e^{-2\kappa t}\bigr).
\end{aligned}
\]
The upper bound in \eqref{eq:rho-explicit-two-sided}
therefore holds for sufficiently large times. For the
lower bound, recall that the covariance eigenvalues remain
bounded and converge to strictly positive limits. Hence
$|\sqrt{\lambda_i(t)}-\sqrt{\lambda_i^\infty}|
\geq c|\lambda_i(t)-\lambda_i^\infty|$
for all sufficiently large $t$, with $c>0$ independent of
$i$. Combining this observation with the covariance lower
bound \eqref{eq:covariance-W2-lower} and the sharp
asymptotic \eqref{eq:covariance-exact-rate-for-W2}
yields
\[
\begin{aligned}
W_2^2(\rho_t,\rho_\infty)
&\geq\sum_{i=1}^d
\bigl(\sqrt{\lambda_i(t)}-\sqrt{\lambda_i^\infty}\bigr)^2\\
&\geq c\max_i|\lambda_i(t)-\lambda_i^\infty|^2
\geq c e^{-2\kappa t}
\end{aligned}
\]
for sufficiently large $t$.  In addition, since $\Sigma_0\neq\Sigma_\infty$, uniqueness of the
covariance equation ensures that $\Sigma_t\neq\Sigma_\infty$
for every finite $t$. Hence $W_2(\rho_t,\rho_\infty)>0$,
and continuity allows the lower bound to be extended
to all $t\geq0$ after reducing $A$.
This completes the proof of \eqref{eq:rho-explicit-two-sided},
from which \eqref{eq:rho-sharp-rate} follows whenever
$\kappa\leq\beta_*$.

For part~(3), assume $\rho_0\in\mathcal P_4(\bbr^d)$.
Let $X_0$ have law $\rho_0$ and put $X:=X_0-m_0$.
The explicit solution representation
\eqref{eq:explicit-solution-diffusion} writes $\rho_t$
as the law of $m_0+\Phi(t,0)X+Z_t$, where
$Z_t\sim\mathcal N(0,Q_t)$ is independent of $X$.
Moreover, the covariance identity
\eqref{eq:covariance-representation-diffusion} gives
$\Sigma_t=Q_t+\Phi(t,0)\Sigma_0\Phi(t,0)^\top$.
As $q_i(t)\to\lambda_i^\infty>0$, the matrices $Q_t$
are uniformly positive definite and bounded for large
$t$. The diagonal representation and
\eqref{eq:ri-explicit-asymptotic} also imply
$\varepsilon_t:=\|\Phi(t,0)\|_{\mathrm{op}}
=\max_i r_i(t,0)\leq Ce^{-\beta_*t}$.

We apply Lemma~\ref{lem:gaussian-smoothing} with
$X=X_0-m_0$, $R=\Phi(t,0)$, $Z=Z_t$, and $Q=Q_t$.
Then $\Sigma_X=\Sigma_0$ and
$\varepsilon=\|R\|_{\mathrm{op}}=\varepsilon_t$,
while the covariance identity above gives
$Q+R\Sigma_XR^\top=\Sigma_t$.
The preceding bounds verify the assumptions of the lemma
for all sufficiently large $t$. Thus, when
$\rho_0\in\mathcal P_4(\bbr^d)$, we obtain
\begin{equation}
\label{eq:rho-matched-Gaussian-upper}
W_2\bigl(\rho_t,\mathcal N(m_0,\Sigma_t)\bigr)
\leq C\varepsilon_t^2\leq Ce^{-2\beta_*t}.
\end{equation}
In particular, this estimate controls the non-Gaussian
part of the solution at twice the transport exponent,
independently of whether the covariance is stationary.

The smoothing estimate controls the distance from $\rho_t$
to the Gaussian measure $\mathcal N(m_0,\Sigma_t)$ having
the same mean and covariance as above. To compare this Gaussian
measure with $\rho_\infty$, we use the common eigenbasis
of $\Sigma_t$ and $\Sigma_\infty$. Coupling the two measures
through the same standard Gaussian vector and using the
positivity of the limiting eigenvalues, we obtain
\[
\begin{aligned}
&W_2\bigl(\mathcal N(m_0,\Sigma_t),\rho_\infty\bigr)\\
&\quad\leq
\left(\sum_{i=1}^d
\bigl(\sqrt{\lambda_i(t)}
-\sqrt{\lambda_i^\infty}\bigr)^2
\right)^{1/2}\\
&\quad\leq
\left(\sum_{i=1}^d
\frac{|\lambda_i(t)-\lambda_i^\infty|^2}
{\lambda_i^\infty}\right)^{1/2}
\leq C\max_i|\lambda_i(t)-\lambda_i^\infty|.
\end{aligned}
\]
If $\Sigma_0=\Sigma_\infty$, uniqueness of the covariance
equation implies $\Sigma_t=\Sigma_\infty$ for all $t\geq0$,
so the Gaussian distance vanishes identically. Otherwise,
\eqref{eq:covariance-exact-rate-for-W2} bounds it by
$Ce^{-\kappa t}$. Combining these observations with
\eqref{eq:rho-matched-Gaussian-upper} and the triangle
inequality gives, for all sufficiently large $t$,
\[
W_2(\rho_t,\rho_\infty)
\leq
\begin{cases}
Ce^{-2\beta_*t},
&\Sigma_0=\Sigma_\infty,\\[1mm]
C\bigl(e^{-2\beta_*t}+e^{-\kappa t}\bigr),
&\Sigma_0\neq\Sigma_\infty,
\end{cases}
\]
which are precisely
\eqref{eq:rho-P4-stationary-upper} and
\eqref{eq:rho-P4-improved-upper}. Since $\rho_t$ is
continuous in $W_2$, both estimates extend to all
$t\geq0$ after increasing $C$.

When $\Sigma_0\neq\Sigma_\infty$ and
$\kappa\leq2\beta_*$, the improved upper bound and
the covariance lower bound in
\eqref{eq:rho-explicit-two-sided} satisfy
\[
Ae^{-\kappa t}
\leq W_2(\rho_t,\rho_\infty)
\leq C\bigl(e^{-2\beta_*t}+e^{-\kappa t}\bigr)
\leq 2Ce^{-\kappa t},
\qquad t\geq0.
\]
This establishes \eqref{eq:rho-P4-sharp-rate}.
For $1\leq n<d$ with non-isotropic initial covariance,
the definitions of $\kappa$ and $\beta_*$ give
\[
\frac{\kappa}{\beta_*}
=\frac{2(d-n)}{d-1}\leq2,
\]
so the condition $\kappa\leq2\beta_*$ always holds
in this case, completing the proof.

\end{proof}

\begin{remark}[Sharpness of the Wasserstein bounds]
\label{rem:sharpness}
The upper and lower Wasserstein bounds do not always have
matching decay exponents. For $1\leq n<d$ with isotropic
initial covariance $\Sigma_0=s_0I_d\neq\lambda_\ast I_d$,
we have $\kappa/\beta_*=2n$. Thus the fourth-moment
estimate gives the sharp rate when $n=1$, but leaves a gap
between the bounds when $n>1$. In the full-dimensional case
$n=d$ with $\Sigma_0\neq\Sigma_\infty$, the relation
\[
\frac{\kappa}{\beta_*}
=2\left(\max_i\lambda_i^\infty\right)
\operatorname{tr}(\Sigma_\infty^{-1})
\geq 2d
\]
shows that the fourth-moment upper bound does not match
the covariance lower bound when $d\geq2$.
The actual rate for compactly supported initial data is identified in
Theorem~\ref{thm:third-order-rates} below under a condition on the third
moments.

This gap between the estimates does not necessarily
reflect the actual convergence rate. Indeed, for Gaussian
initial data, the solution remains Gaussian, and the
common eigenbasis of $\Sigma_t$ and $\Sigma_\infty$ gives
\[
W_2^2(\rho_t,\rho_\infty)
=\sum_{i=1}^d
\bigl(\sqrt{\lambda_i(t)}
-\sqrt{\lambda_i^\infty}\bigr)^2.
\]
Consequently, whenever $\Sigma_0\neq\Sigma_\infty$, the
sharp covariance asymptotics imply
$W_2(\rho_t,\rho_\infty)\asymp e^{-\kappa t}$,
regardless of the relation between $\kappa$ and $\beta_*$.
\end{remark}

\begin{remark}[Initial data with finite second moment]
\label{rem:P2-little-o}
For initial data in $\mathcal P_2(\bbr^d)$, the first assertion of
Lemma~\ref{lem:gaussian-smoothing} still improves on the direct
coupling. Applying it with $X=X_0-m_0$, $R=\Phi(t,0)$, and $Q=Q_t$, as
in the proof of part~(3) of Theorem~\ref{thm:explicit-W2-decay}, and
using $\|\Phi(t,0)\|_{\mathrm{op}}\leq Ce^{-\beta_*t}$, we obtain
\[
W_2\bigl(\rho_t,\mathcal N(m_0,\Sigma_t)\bigr)
=o\bigl(e^{-\beta_*t}\bigr)
\qquad\text{as }t\to\infty.
\]
Hence the term $e^{-\beta_*t}$ in the upper bounds
\eqref{eq:rho-stationary-covariance-upper} and
\eqref{eq:rho-explicit-two-sided} can be replaced by
$o(e^{-\beta_*t})$.
\end{remark}

\begin{proof}[Proof of Theorem~\ref{thm:main-longtime}]
The classification of Gaussian minimizers in
Section~\ref{sec:minimizers}, together with
Proposition~\ref{prop:minimizers-stationary}, establishes
the characterization of stationary weak solutions in
part~(1). Propositions~\ref{prop:covariance-nd} and
\ref{prop:covariance-nlessd} determine the limiting
covariance $\Sigma_\infty$ in both regimes. Theorem~
\ref{thm:explicit-W2-decay} then identifies the
corresponding Gaussian equilibrium and establishes
exponential convergence to it in $W_2$ for every
$\rho_0\in\mathcal P_2(\bbr^d)$, proving part~(2).

To establish part~(3), suppose that $1\leq n<d$ and
$\rho_0\in\mathcal P_4(\bbr^d)$ has non-isotropic covariance.
Since $\Sigma_\infty=\lambda_\ast I_d$, we have
$\Sigma_0\neq\Sigma_\infty$. In this case, the decay
exponents defined in Theorem~\ref{thm:explicit-W2-decay}
satisfy
\[
\beta_*=\frac{\gamma}{\lambda_\ast},
\qquad
\kappa=\frac{2\gamma(d-n)}{(d-1)\lambda_\ast},
\qquad
\frac{\kappa}{2\beta_*}
=\frac{d-n}{d-1}\leq1.
\]
The fourth-moment upper bound
\eqref{eq:rho-P4-improved-upper} and the covariance
lower bound \eqref{eq:rho-explicit-two-sided} can
therefore be combined to obtain the two-sided estimate
\[
Ae^{-\kappa t}
\leq W_2(\rho_t,\rho_\infty)
\leq C\bigl(e^{-2\beta_*t}+e^{-\kappa t}\bigr)
\leq 2Ce^{-\kappa t},
\qquad t\geq0,
\]
which gives the sharp exponential rate asserted in
part~(3).
\end{proof}

\subsection{Sharp rates for compactly supported initial data}
\label{subsec:third-order}

Theorem~\ref{thm:explicit-W2-decay} identifies the sharp Wasserstein
rate only when covariance relaxation is not faster than twice the
transport exponent. In the remaining cases, the decay of the
non-Gaussian part of the solution has to be quantified more precisely.
For Gaussian initial data this part is absent, and the rate is $\kappa$
by Remark~\ref{rem:sharpness}. In this subsection, we show
that, for compactly supported
initial data, the non-Gaussian part decays at least as fast as
$e^{-3\beta_*t}$, and that this rate is attained whenever the third moments of
the initial datum do not vanish in the slowest directions.

The upper bound relies on the following refinement of
Lemma~\ref{lem:gaussian-smoothing} for bounded perturbations, whose
proof is also provided in Appendix~\ref{app:gaussian-smoothing}.

\begin{lemma}[Third-order Gaussian smoothing estimate]
\label{lem:third-order-smoothing}
Let $X$ be a centered random vector in $\mathbb R^d$ satisfying
$|X|\leq M$ almost surely for some $M>0$, and write
$\Sigma_X:=\mathbb E[XX^\top]$. Fix constants $0<q_-<q_+<\infty$.
Let $R$ be a deterministic $d\times d$ matrix with
$\varepsilon:=\|R\|_{\mathrm{op}}$, and let $Z$ be independent of $X$
with Gaussian distribution $\mathcal N(0,Q)$, where
\[
q_-I_d\leq Q\leq q_+I_d.
\]
Then there exist constants $\varepsilon_0>0$ and $C>0$, depending only
on $d$, $M$, $q_-$, and $q_+$, such that
\begin{equation}
\label{eq:third-order-smoothing-W2}
W_2\left(
\operatorname{Law}(RX+Z),
\mathcal N(0,Q+R\Sigma_XR^\top)
\right)
\leq C\varepsilon^3
\qquad\text{whenever }\varepsilon\leq\varepsilon_0.
\end{equation}
If, in addition, all third moments of $X$ vanish, that is,
$\mathbb E[X_iX_jX_k]=0$ for all $1\leq i,j,k\leq d$, then
\eqref{eq:third-order-smoothing-W2} holds with $\varepsilon^4$ in place
of $\varepsilon^3$.
\end{lemma}

For bounded $X$, Lemma~\ref{lem:third-order-smoothing} also yields the
estimate \eqref{eq:gaussian-smoothing-W2} of
Lemma~\ref{lem:gaussian-smoothing}, because the left-hand side of
\eqref{eq:gaussian-smoothing-W2} is bounded uniformly in $R$ and $Q$
when $\varepsilon\leq1$. For compactly supported initial data, part~(3)
of Theorem~\ref{thm:explicit-W2-decay} is therefore independent of
Theorem~\ref{thm:paulin-whalley-convolution}, which is needed only to
reach initial data with finite fourth moment.

For $\mu\in\mathcal P_4(\bbr^d)$, $m\in\bbr^d$, and
$\xi_1,\xi_2,\xi_3\in\bbr^d$, we write
\[
\mathcal T_m[\mu](\xi_1,\xi_2,\xi_3)
:=\int_{\bbr^d}\prod_{\ell=1}^3\xi_\ell\cdot(x-m)\,\d\mu(x).
\]
The next lemma, which is used for the lower bound, shows
that this trilinear form is controlled by the Wasserstein distance when
fourth moments are bounded.

\begin{lemma}[Third moments and the Wasserstein distance]
\label{lem:third-moment-lower}
Let $\mu,\nu\in\mathcal P_4(\bbr^d)$, let $m\in\bbr^d$, and let
$\xi_1,\xi_2,\xi_3\in\bbr^d$ be unit vectors. Then
\[
\begin{aligned}
&\bigl|\mathcal T_m[\mu](\xi_1,\xi_2,\xi_3)
-\mathcal T_m[\nu](\xi_1,\xi_2,\xi_3)\bigr|\\
&\qquad\leq
3W_2(\mu,\nu)
\left(
\int_{\bbr^d}|x-m|^4\,\d\mu(x)
+\int_{\bbr^d}|x-m|^4\,\d\nu(x)
\right)^{1/2}.
\end{aligned}
\]
\end{lemma}

\begin{proof}
Let $(Y,\overline Y)$ be an optimal coupling of $\mu$ and $\nu$ for the
quadratic cost, and set $a:=Y-m$ and $b:=\overline Y-m$. The
telescoping identity
\[
\begin{aligned}
\prod_{\ell=1}^3\xi_\ell\cdot a-\prod_{\ell=1}^3\xi_\ell\cdot b
={}&\bigl(\xi_1\cdot(a-b)\bigr)(\xi_2\cdot a)(\xi_3\cdot a)\\
&+(\xi_1\cdot b)\bigl(\xi_2\cdot(a-b)\bigr)(\xi_3\cdot a)\\
&+(\xi_1\cdot b)(\xi_2\cdot b)\bigl(\xi_3\cdot(a-b)\bigr)
\end{aligned}
\]
and $|\xi_\ell|=1$ give
\[
\left|
\prod_{\ell=1}^3\xi_\ell\cdot a-\prod_{\ell=1}^3\xi_\ell\cdot b
\right|
\leq|a-b|\bigl(|a|^2+|a||b|+|b|^2\bigr)
\leq\frac32|a-b|\bigl(|a|^2+|b|^2\bigr).
\]
Taking expectations and applying the Cauchy--Schwarz inequality, we
obtain
\[
\begin{aligned}
&\bigl|\mathcal T_m[\mu](\xi_1,\xi_2,\xi_3)
-\mathcal T_m[\nu](\xi_1,\xi_2,\xi_3)\bigr|\\
&\qquad\leq
\frac32\bigl(\mathbb E|Y-\overline Y|^2\bigr)^{1/2}
\Bigl(\mathbb E\bigl(|a|^2+|b|^2\bigr)^2\Bigr)^{1/2}\\
&\qquad\leq
\frac{3}{\sqrt2}W_2(\mu,\nu)
\bigl(\mathbb E|a|^4+\mathbb E|b|^4\bigr)^{1/2},
\end{aligned}
\]
which implies the claim.
\end{proof}

To state the main result of this subsection, we retain the
eigenbasis fixed in
\eqref{eq:initial-covariance-eigenvalues} and write $p_i:=P^\top e_i$
for the $i$-th eigenvector, where $e_i$ denotes the $i$-th standard
basis vector of $\bbr^d$. By \eqref{eq:Phi-diagonal-diffusion}, each
$p_i$ is an eigenvector of $\Phi(t,0)$ with eigenvalue $r_i(t,0)$.

\begin{theorem}[Sharp rates for compactly supported data]
\label{thm:third-order-rates}
Let $\rho_0$ be a compactly supported probability measure on
$\bbr^d$, and let
$\rho_\infty$, $a_i^\infty$, $\beta_*$, and $\kappa$ be as in
Theorem~\ref{thm:explicit-W2-decay}. Define
\[
I_*:=\{1\leq i\leq d:\ a_i^\infty=\beta_*\},
\qquad
E_*:=\Span\{p_i:\ i\in I_*\},
\]
so that $E_*$ is the eigenspace of $\Sigma_\infty$ associated with its
largest eigenvalue. The following assertions hold.
\begin{enumerate}
\item There exists a constant $C>0$ such that
\begin{equation}
\label{eq:third-order-matched}
W_2\bigl(\rho_t,\mathcal N(m_0,\Sigma_t)\bigr)
\leq Ce^{-3\beta_*t},
\qquad t\geq0.
\end{equation}
Consequently,
\begin{equation}
\label{eq:third-order-upper}
W_2(\rho_t,\rho_\infty)
\leq
\begin{cases}
Ce^{-3\beta_*t},
&\Sigma_0=\Sigma_\infty,\\[1mm]
C\bigl(e^{-3\beta_*t}+e^{-\kappa t}\bigr),
&\Sigma_0\neq\Sigma_\infty,
\end{cases}
\qquad t\geq0.
\end{equation}

\item Suppose that
\begin{equation}
\label{eq:third-moment-condition}
\mathcal T_{m_0}[\rho_0](\xi,\xi,\xi)
=\int_{\bbr^d}\bigl(\xi\cdot(x-m_0)\bigr)^3\,\d\rho_0(x)\neq0
\qquad\text{for some }\xi\in E_*.
\end{equation}
Then there exists a constant $A>0$ such that
\begin{equation}
\label{eq:third-order-lower}
W_2(\rho_t,\rho_\infty)\geq Ae^{-3\beta_*t},
\qquad t\geq0.
\end{equation}

\item Under the assumption of part~(2), the Wasserstein distance has
the sharp exponential rate
\begin{equation}
\label{eq:third-order-sharp}
W_2(\rho_t,\rho_\infty)\asymp
\begin{cases}
e^{-3\beta_*t},
&\Sigma_0=\Sigma_\infty,\\[1mm]
e^{-\min\{\kappa,3\beta_*\}t},
&\Sigma_0\neq\Sigma_\infty.
\end{cases}
\end{equation}
When $\Sigma_0\neq\Sigma_\infty$ and $\kappa\leq3\beta_*$, the relation
\eqref{eq:third-order-sharp} holds without assuming
\eqref{eq:third-moment-condition}.
\end{enumerate}
\end{theorem}

\begin{proof}
Let $X_0$ have law $\rho_0$ and put $X:=X_0-m_0$. As in the proof of
Theorem~\ref{thm:explicit-W2-decay}, the measure $\rho_t$ is the law of
$m_0+\Phi(t,0)X+Z_t$, where $Z_t\sim\mathcal N(0,Q_t)$ is independent of
$X$, and
\[
\Sigma_t=Q_t+\Phi(t,0)\Sigma_0\Phi(t,0)^\top,
\qquad
\varepsilon_t:=\|\Phi(t,0)\|_{\mathrm{op}}\leq Ce^{-\beta_*t}.
\]

For part~(1), note that $X$ is bounded, since $\rho_0$ is compactly
supported. The matrices $Q_t$ are uniformly positive definite and
bounded for large $t$, and $\varepsilon_t\to0$ as $t\to\infty$. We may
therefore apply Lemma~\ref{lem:third-order-smoothing} with
$R=\Phi(t,0)$ and $Q=Q_t$ for all sufficiently large $t$, which gives
\[
W_2\bigl(\rho_t,\mathcal N(m_0,\Sigma_t)\bigr)
\leq C\varepsilon_t^3\leq Ce^{-3\beta_*t}.
\]
Since the second moments of $\rho_t$ and $\mathcal N(m_0,\Sigma_t)$ are
bounded on every finite time interval, the estimate extends to all
$t\geq0$ after increasing $C$, which proves
\eqref{eq:third-order-matched}. In addition, as shown in the proof of
Theorem~\ref{thm:explicit-W2-decay}, the distance between
$\mathcal N(m_0,\Sigma_t)$ and $\rho_\infty$ vanishes identically when
$\Sigma_0=\Sigma_\infty$ and is bounded by $Ce^{-\kappa t}$ otherwise.
The triangle inequality then yields \eqref{eq:third-order-upper}.

For part~(2), we first observe that the trilinear form
$\mathcal T_{m_0}[\rho_0]$ does not vanish on
$E_*\times E_*\times E_*$ by \eqref{eq:third-moment-condition}. Since
$E_*$ is spanned by the vectors $p_i$ with $i\in I_*$, there exist
indices $i,j,k\in I_*$ such that
\[
\tau:=\mathcal T_{m_0}[\rho_0](p_i,p_j,p_k)
=\mathbb E\bigl[(p_i\cdot X)(p_j\cdot X)(p_k\cdot X)\bigr]\neq0.
\]
Using $p_\ell\cdot\Phi(t,0)X=r_\ell(t,0)\,p_\ell\cdot X$, we expand
\[
\begin{aligned}
\mathcal T_{m_0}[\rho_t](p_i,p_j,p_k)
=\mathbb E\Bigl[
&\bigl(r_i(t,0)\,p_i\cdot X+p_i\cdot Z_t\bigr)
\bigl(r_j(t,0)\,p_j\cdot X+p_j\cdot Z_t\bigr)\\
&\times\bigl(r_k(t,0)\,p_k\cdot X+p_k\cdot Z_t\bigr)\Bigr].
\end{aligned}
\]
The terms containing exactly one or exactly two factors of $Z_t$
vanish, because $X$ and $Z_t$ are independent and centered. The term
containing three factors of $Z_t$ vanishes, because the third moments
of a centered Gaussian vector are zero. Hence
\[
\mathcal T_{m_0}[\rho_t](p_i,p_j,p_k)
=r_i(t,0)r_j(t,0)r_k(t,0)\,\tau.
\]
Moreover, since $\rho_\infty$ is a Gaussian measure with mean
$m_0$, its third central moments vanish, and thus
$\mathcal T_{m_0}[\rho_\infty](p_i,p_j,p_k)=0$.
The fourth moments of $\rho_t$ about $m_0$ are bounded uniformly in
$t\geq0$. Indeed, $\|\Phi(t,0)\|_{\mathrm{op}}\leq1$, while
$Q_t\leq\Sigma_t$ and the covariance $\Sigma_t$ is bounded on
$[0,\infty)$, so that
\[
\int_{\bbr^d}|x-m_0|^4\,\d\rho_t(x)
\leq8\bigl(\mathbb E|X|^4+\mathbb E|Z_t|^4\bigr)
\leq8\mathbb E|X|^4+24\bigl(\operatorname{tr}\Sigma_t\bigr)^2
\leq C.
\]
Applying Lemma~\ref{lem:third-moment-lower} with $\mu=\rho_t$,
$\nu=\rho_\infty$, and $m=m_0$, we obtain
\[
|\tau|\,r_i(t,0)r_j(t,0)r_k(t,0)
\leq CW_2(\rho_t,\rho_\infty),
\qquad t\geq0.
\]
Finally, the asymptotic relation \eqref{eq:ri-explicit-asymptotic}
and $a_\ell^\infty=\beta_*$ for $\ell\in I_*$ give
$r_\ell(t,0)\geq ce^{-\beta_*t}$ for every $t\geq0$ and
$\ell\in\{i,j,k\}$, with some $c>0$. This proves
\eqref{eq:third-order-lower}.

For part~(3), suppose first that $\Sigma_0=\Sigma_\infty$. Then
\eqref{eq:third-order-upper} and \eqref{eq:third-order-lower} have the
same exponential order. If $\Sigma_0\neq\Sigma_\infty$, the covariance
lower bound in \eqref{eq:rho-explicit-two-sided} and
\eqref{eq:third-order-lower} give
\[
W_2(\rho_t,\rho_\infty)
\geq A\max\bigl\{e^{-\kappa t},e^{-3\beta_*t}\bigr\}
=Ae^{-\min\{\kappa,3\beta_*\}t},
\]
while \eqref{eq:third-order-upper} gives
\[
W_2(\rho_t,\rho_\infty)
\leq C\bigl(e^{-3\beta_*t}+e^{-\kappa t}\bigr)
\leq2Ce^{-\min\{\kappa,3\beta_*\}t}.
\]
When $\kappa\leq3\beta_*$, the covariance lower bound alone matches the
upper bound, and the assumption \eqref{eq:third-moment-condition} is
not needed.
\end{proof}

\begin{remark}[Regimes of the sharp rate]
\label{rem:regimes}
Combining Theorem~\ref{thm:third-order-rates} with the ratios computed
in Theorem~\ref{thm:explicit-W2-decay} and
Remark~\ref{rem:sharpness}, we obtain the following picture for
compactly supported initial data with $\Sigma_0\neq\Sigma_\infty$.
\begin{enumerate}
\item If $1\leq n<d$ and $\Sigma_0$ is not isotropic, then
$\kappa\leq2\beta_*$, and the rate is $\kappa$.

\item If $1\leq n<d$ and $\Sigma_0=s_0I_d$ with $s_0\neq\lambda_\ast$,
then $\kappa=2n\beta_*$. The rate is $\kappa$ when $n=1$. When
$n\geq2$, it is $3\beta_*=3\gamma/\lambda_\ast$ under
\eqref{eq:third-moment-condition}.

\item If $n=d$, then $\kappa\geq2d\beta_*$. The rate is $\kappa$ when
$d=1$. When $d\geq2$, it is
$3\beta_*=3\gamma/\max_i\lambda_i^\infty$ under
\eqref{eq:third-moment-condition}.
\end{enumerate}
In the subcritical case $1\leq n<d$, we have $E_*=\bbr^d$, and
\eqref{eq:third-moment-condition} only requires that the third central
moments of $\rho_0$ do not vanish identically. In the full-dimensional
case $n=d$, the space $E_*$ is the eigenspace of $\Sigma_0$ associated
with its largest eigenvalue. In the regimes where the rate is
$3\beta_*$, the covariance relaxes strictly faster than the measure
itself, and the long-time behavior is dictated by the skewness of the
initial datum.
\end{remark}

\begin{remark}[Vanishing third moments]
If all third central moments of $\rho_0$ vanish, for instance when
$\rho_0$ is symmetric with respect to $m_0$, the second assertion of
Lemma~\ref{lem:third-order-smoothing} shows that
\eqref{eq:third-order-matched} and \eqref{eq:third-order-upper} hold
with $4\beta_*$ in place of $3\beta_*$. A matching lower bound would
require an analysis of the fourth-order cumulants of $\rho_0$, which we
do not pursue here.
\end{remark}

\begin{proof}[Proof of Theorem~\ref{thm:main-third-order}]
The exponent $\beta_*$ and the space $E_*$ in the statement coincide
with those of Theorem~\ref{thm:third-order-rates}, since
$a_i^\infty=\gamma/\lambda_i^\infty$. Part~(1) is the second case of
\eqref{eq:third-order-upper}, and part~(2) follows from
\eqref{eq:third-order-lower} as in the proof of part~(3) of
Theorem~\ref{thm:third-order-rates}. For part~(3), if $n=d\geq2$, then
$\kappa\geq2d\beta_*\geq4\beta_*$ by Remark~\ref{rem:sharpness}. If
$2\leq n<d$ and $\Sigma_0$ is isotropic, then
$\kappa=2n\beta_*\geq4\beta_*$. In both cases $\kappa>3\beta_*$, and
the rate in part~(2) equals $3\beta_*$.
\end{proof}

\subsection{Numerical illustration}
\label{subsec:numerics}

We illustrate Theorem~\ref{thm:third-order-rates} in two cases where
covariance relaxation is faster than the decay of the non-Gaussian
part. If $\rho_0$ is a product measure in the eigenbasis of $\Sigma_0$,
the representation \eqref{eq:explicit-solution-diffusion}, together
with \eqref{eq:Phi-diagonal-diffusion} and
\eqref{eq:Q-diagonal-diffusion}, shows that $\rho_t$ and $\rho_\infty$
are product measures in the same basis. The squared Wasserstein
distance is then the sum of the squared distances between the
one-dimensional marginals. These can be computed from the quantile
functions of the marginals once the eigenvalue system
\eqref{eq:eigenvalue-diffusion} has been integrated numerically.

\begin{figure}[!ht]
\centering
\includegraphics[width=\textwidth]{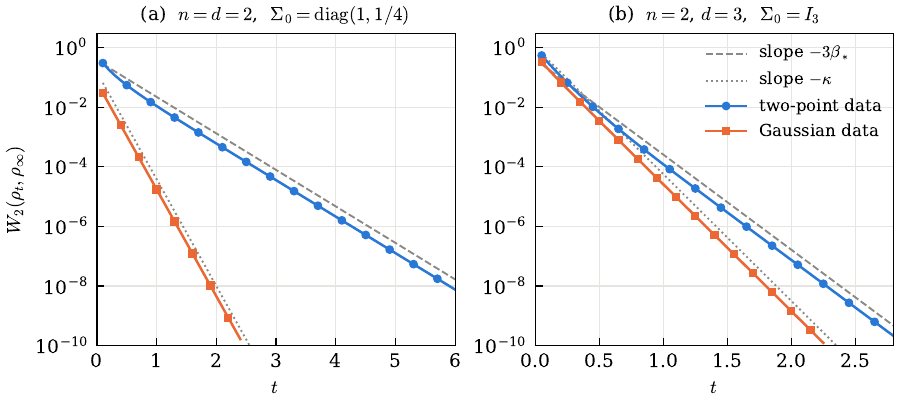}
\caption{Decay of $W_2(\rho_t,\rho_\infty)$ for $\gamma=1$ and product
initial data. Gaussian data relax at the covariance rate $\kappa$,
whereas two-point data with nonvanishing third moment relax at the
slower rate $3\beta_*$. The dashed and dotted lines indicate the
slopes $-3\beta_*$ and $-\kappa$, respectively.}
\label{fig:rates}
\end{figure}

Figure~\ref{fig:rates} shows two examples with $\gamma=1$. In the
first, $n=d=2$ and $\Sigma_0=\operatorname{diag}(1,1/4)$, which gives
$\Sigma_\infty\approx\operatorname{diag}(1.063,0.313)$,
$3\beta_*\approx2.82$, and $\kappa\approx8.26$. In the second, $n=2$,
$d=3$, and $\Sigma_0=I_3$, which gives $\lambda_\ast=1/\sqrt6$,
$3\beta_*\approx7.35$, and $\kappa\approx9.80$. In each case, we
compare Gaussian initial data with a product of two-point measures
having the same covariance. The first marginal of the latter is
$\frac23\delta_{-a}+\frac13\delta_{2a}$ with $a=1/\sqrt2$, which has
unit variance and a nonvanishing third moment, and the remaining
marginals are symmetric two-point measures. For Gaussian data, the
distance decays at the covariance rate $\kappa$, in accordance with
Remark~\ref{rem:sharpness}. For the two-point data, it decays at the
slower rate $3\beta_*$ predicted by
Theorem~\ref{thm:third-order-rates}. A least-squares fit of the
computed distances gives the exponents $2.82$ and $7.35$ for the
two-point data, and $8.26$ and $9.80$ for the Gaussian data.

\section*{Declarations}

\noindent\textbf{Competing interests.} The authors declare that they have no competing interests.

\noindent\textbf{Data availability.} No datasets were generated or analyzed during the current study.

\noindent\textbf{Generative AI and AI-assisted technologies in the manuscript preparation process.} The authors used ChatGPT (OpenAI) to assist with language editing
and manuscript organization, and to obtain feedback on mathematical
arguments and references.

\appendix

\section{Cutoff arguments for the moment identities}
\label{app:moment-cutoff}
\setcounter{equation}{0}

This appendix supplies the details omitted in the proof of
Proposition~\ref{prop:moment-equations-diffusion}. We repeatedly use the fact
that continuity of $t\mapsto\rho_t$ in $W_2$ implies continuity of the first
and second moments. See, for example, \cite[Theorem~6.9]{villani2009optimal}.

\begin{proof}[Detailed proof of Proposition~\ref{prop:moment-equations-diffusion}]
Since the map $t\mapsto\rho_t$ is continuous with respect to $W_2$, the second moments are uniformly bounded on $[0,T]$. More precisely,
\begin{equation}
\label{eq:uniform-second-moment}
\sup_{0\leq t\leq T}\int_{\bbr^\dm}|x|^2\,\d\rho_t(x)<\infty.
\end{equation}
The maps $t\mapsto m_t$ and $t\mapsto\Sigma_t$ are continuous because convergence in $W_2$ implies convergence of the first and second moments. Consequently, $A_t:=2c_nT_{n-1}(\Sigma_t)$ is bounded on $[0,T]$. In particular, there exists a constant $C_T>0$ such that
\begin{equation}
\label{eq:linear-growth-drift}
|v[\rho_t](x)|=|A_t(x-m_t)|\leq C_T(1+|x|)
\end{equation}
for every $x\in\bbr^\dm$ and $t\in[0,T]$.

We start with the mean. Let $\eta\in C_c^\infty(\bbr^\dm)$ satisfy
\[
0\leq\eta\leq1,
\qquad
\eta(x)=1\quad\text{for }|x|\leq1,
\qquad
\eta(x)=0\quad\text{for }|x|\geq2.
\]
For $R>1$ and $j\in\{1,\ldots,\dm\}$, set
$\eta_R(x):=\eta(x/R)$ and $\varphi_{R,j}(x):=x_j\eta_R(x)$. Then
\[
\begin{aligned}
\nabla\varphi_{R,j}(x)
&=e_j\eta_R(x)+x_j\nabla\eta_R(x),\\
\Delta\varphi_{R,j}(x)
&=2\partial_j\eta_R(x)+x_j\Delta\eta_R(x).
\end{aligned}
\]
Since the derivatives of $\eta_R$ are supported in
$\{R\leq|x|\leq2R\}$ and
\[
|\nabla\eta_R(x)|\leq\frac{C}{R},
\qquad
|\Delta\eta_R(x)|\leq\frac{C}{R^2},
\]
we have
\[
|\nabla\varphi_{R,j}(x)|\leq C,
\qquad
|\Delta\varphi_{R,j}(x)|
\leq\frac{C}{R}\indc_{\{R\leq|x|\leq2R\}}.
\]
Using the test function
$\varphi(t,x)=\zeta(t)\varphi_{R,j}(x)$ with
$\zeta\in C_c^\infty(0,T)$ in the weak formulation, we obtain
\[
\begin{aligned}
\frac{\d}{\d t}\int_{\bbr^\dm}\varphi_{R,j}(x)\,\d\rho_t(x)
={}&\int_{\bbr^\dm}\nabla\varphi_{R,j}(x)\cdot v[\rho_t](x)\,\d\rho_t(x)
\\
&+\gamma\int_{\bbr^\dm}\Delta\varphi_{R,j}(x)\,\d\rho_t(x)
\end{aligned}
\]
in the sense of distributions on $(0,T)$.

As $R\to\infty$, we have
\[
\varphi_{R,j}(x)\longrightarrow x_j,
\qquad
\nabla\varphi_{R,j}(x)\longrightarrow e_j
\]
pointwise. Moreover,
\[
|\varphi_{R,j}(x)|\leq|x|,
\qquad
|\nabla\varphi_{R,j}(x)\cdot v[\rho_t](x)|
\leq C_T(1+|x|).
\]
The bounds \eqref{eq:uniform-second-moment} and \eqref{eq:linear-growth-drift}
allow us to pass to the limit in the terms containing
$\varphi_{R,j}$ and $\nabla\varphi_{R,j}$ by dominated convergence.
For the diffusion term, the preceding estimate for
$\Delta\varphi_{R,j}$ gives
\[
\begin{aligned}
&\left|
\int_0^T\zeta(t)
\int_{\bbr^\dm}\Delta\varphi_{R,j}(x)\,\d\rho_t(x)\,\d t
\right|
\\
&\qquad\leq
\frac{C}{R}\|\zeta\|_{L^\infty(0,T)}
\int_0^T\rho_t\bigl(\{R\leq|x|\leq2R\}\bigr)\,\d t
\\
&\qquad\leq
\frac{CT}{R}\|\zeta\|_{L^\infty(0,T)}
\longrightarrow0.
\end{aligned}
\]
We therefore obtain
\[
\frac{\d}{\d t}\int_{\bbr^\dm}x_j\,\d\rho_t(x)
=\int_{\bbr^\dm}e_j\cdot v[\rho_t](x)\,\d\rho_t(x).
\]
Since this holds for every $j$,
\[
\frac{\d}{\d t}m_t
=\int_{\bbr^\dm}v[\rho_t](x)\,\d\rho_t(x).
\]
Using the explicit form of the velocity field, we find
\[
\int_{\bbr^\dm}v[\rho_t](x)\,\d\rho_t(x)
=-A_t\int_{\bbr^\dm}(x-m_t)\,\d\rho_t(x)
=0.
\]
Thus $m_t=m_0$ for every $t\in[0,T]$. The covariance can therefore be written as
\[
\Sigma_t
=\int_{\bbr^\dm}(x-m_0)(x-m_0)^\top\,\d\rho_t(x).
\]
Fix $i,j\in\{1,\ldots,\dm\}$ and set
\[
\Psi_{ij}(x):=(x_i-m_{0,i})(x_j-m_{0,j}),
\qquad
\Psi_{ij}^R(x):=\Psi_{ij}(x)\eta_R(x).
\]
The untruncated function satisfies
\[
\nabla\Psi_{ij}(x)
=(x_j-m_{0,j})e_i+(x_i-m_{0,i})e_j,
\qquad
\Delta\Psi_{ij}(x)=2\delta_{ij}.
\]
The cutoff derivatives are
\[
\begin{aligned}
\nabla\Psi_{ij}^R
&=\eta_R\nabla\Psi_{ij}+\Psi_{ij}\nabla\eta_R,\\
\Delta\Psi_{ij}^R
&=\eta_R\Delta\Psi_{ij}
+2\nabla\Psi_{ij}\cdot\nabla\eta_R
+\Psi_{ij}\Delta\eta_R.
\end{aligned}
\]
On the annulus $\{R\leq|x|\leq2R\}$, the last two terms in
$\Delta\Psi_{ij}^R$ are bounded by
$C(1+|x|^2)/R^2+C(1+|x|)/R$. The drift cutoff produces the additional error
$\Psi_{ij}\nabla\eta_R\cdot v[\rho_t]$. Using
$|v[\rho_t](x)|\leq C_T(1+|x|)$ and $|x|/R\leq2$ on the annulus, we have
\[
\bigl|\Psi_{ij}(x)\nabla\eta_R(x)\cdot v[\rho_t](x)\bigr|
\leq C_T(1+|x|^2)\mathbf 1_{\{R\leq|x|\leq2R\}}.
\]
The diffusion cutoff errors admit the same type of integrable domination.
Together with \eqref{eq:uniform-second-moment}, this domination allows us to apply dominated convergence with respect to $\d t\,\d\rho_t$, and all cutoff errors vanish as $R\to\infty$.
Applying the weak formulation to $\zeta(t)\Psi_{ij}^R(x)$ and passing to the limit yields
\[
\frac{\d}{\d t}(\Sigma_t)_{ij}
=\int_{\bbr^\dm}
\left[
(x_j-m_{0,j})v_i[\rho_t](x)
+(x_i-m_{0,i})v_j[\rho_t](x)
\right]
\,\d\rho_t(x)
+2\gamma\delta_{ij}.
\]
In matrix form, this reads
\[
\frac{\d}{\d t}\Sigma_t
=\int_{\bbr^\dm}
\left[
(x-m_0)v[\rho_t](x)^\top
+v[\rho_t](x)(x-m_0)^\top
\right]
\,\d\rho_t(x)
+2\gamma I_\dm.
\]
Substituting $v[\rho_t](x)=-A_t(x-m_0)$ gives
\[
\frac{\d}{\d t}\Sigma_t
=-\Sigma_tA_t-A_t\Sigma_t+2\gamma I_\dm.
\]
Since $T_{n-1}(\Sigma_t)$ is a polynomial in $\Sigma_t$, it commutes with $\Sigma_t$. Substituting its definition therefore yields
\[
\frac{\d}{\d t}\Sigma_t
=-4c_n\Sigma_tT_{n-1}(\Sigma_t)+2\gamma I_\dm,
\]
where the right-hand side is continuous on $[0,T]$. Since $t\mapsto\Sigma_t$ is continuous on $[0,T]$, the distributional identity shows that $\Sigma_t$ is continuously differentiable on $[0,T]$ and that the equation holds for every $t\in[0,T]$.
\end{proof}

\section{Duality for the linear Fokker--Planck equation}
\label{app:linear-duality}

We justify the duality step used in the proof of
Theorem~\ref{thm:global-wellposedness-diffusion}.

\begin{lemma}[Backward duality identity]
Let $A\in C([0,T];\mathbb R^{\dm\times\dm})$ be symmetric. For $0\leq s\leq r\leq T$, let $\Phi(r,s)$ solve
\[
\partial_r\Phi(r,s)=-A_r\Phi(r,s),
\qquad
\Phi(s,s)=I_\dm.
\]
We associate with this evolution family the covariance matrix
\[
Q(r,s):=2\gamma\int_s^r\Phi(r,\tau)\Phi(r,\tau)^\top\,\d\tau.
\]
Let $\mu\in C([0,T];(\mathcal P_2(\bbr^\dm),W_2))$ be a weak solution of
\[
\partial_t\mu_t
=\nabla\cdot\left(A_t(x-m_0)\mu_t\right)+\gamma\Delta\mu_t,
\qquad
\mu_{t=0}=\rho_0.
\]
Fix $t\in(0,T]$ and $\psi\in C_c^\infty(\bbr^\dm)$. Define
\[
u(s,x)
:=\int_{\bbr^\dm}
\psi\big(m_0+\Phi(t,s)(x-m_0)+z\big)
\,\d\mathcal N(0,Q(t,s))(z),
\qquad 0\leq s\leq t.
\]
The duality identity to be proved is
\[
\int_{\bbr^\dm}\psi(x)\,\d\mu_t(x)
=\int_{\bbr^\dm}u(0,x)\,\d\rho_0(x).
\]
\end{lemma}

\begin{proof}
The function $u$ satisfies $u(t,x)=\psi(x)$. We verify directly that it
solves the backward Kolmogorov equation. For fixed $t$, set
\[
F_s:=\Phi(t,s),
\qquad
Q_s:=Q(t,s),
\qquad
y_s(x):=m_0+F_s(x-m_0).
\]
The evolution family and the definition of $Q(t,s)$ give
\[
\partial_sF_s=F_sA_s,
\qquad
\partial_sQ_s=-2\gamma F_sF_s^\top.
\]
For $0\leq s<t$, the matrix $Q_s$ is positive definite. We may therefore
differentiate the Gaussian expectation with respect to its mean and covariance.
The resulting identity extends to $s=t$ by continuity, and differentiation gives
\begin{align*}
\partial_su(s,x)
={}&F_sA_s(x-m_0)\cdot
\int_{\bbr^\dm}\nabla\psi(y_s(x)+z)\,\d\mathcal N(0,Q_s)(z)
\\
&-\gamma\operatorname{tr}\left(
F_sF_s^\top
\int_{\bbr^\dm}D^2\psi(y_s(x)+z)\,\d\mathcal N(0,Q_s)(z)
\right).
\end{align*}
The spatial derivatives are
\[
\begin{aligned}
\nabla u(s,x)
&=F_s^\top
\int_{\bbr^\dm}\nabla\psi(y_s(x)+z)\,\d\mathcal N(0,Q_s)(z),\\
\Delta u(s,x)
&=\operatorname{tr}\left(
F_sF_s^\top
\int_{\bbr^\dm}D^2\psi(y_s(x)+z)\,\d\mathcal N(0,Q_s)(z)
\right).
\end{aligned}
\]
Since $A_s$ is symmetric, the first term in $\partial_su$ is $A_s(x-m_0)\cdot\nabla u$. The backward equation therefore becomes
\[
\partial_su(s,x)
-A_s(x-m_0)\cdot\nabla u(s,x)
+\gamma\Delta u(s,x)=0.
\]
Since $\psi$ is smooth and compactly supported, $u$, $\nabla u$, and $D^2u$
are bounded on $[0,t]\times\bbr^\dm$.

The function $u$ need not be compactly supported. Let $\eta_R$ be the cutoff
used in Appendix~\ref{app:moment-cutoff} and set
\[
u_R(s,x):=u(s,x)\eta_R(x).
\]
Although $A$ is only continuous in time, the formulas above give continuous $u_R$, $\partial_su_R$, $\nabla u_R$, and $\Delta u_R$ on each compact space-time cylinder. Choose smooth cutoffs $\chi_\delta$ supported in $[0,t)$, equal to one on $[0,t-2\delta]$, and vanishing on $[t-\delta,t]$. After extending $\chi_\delta u_R$ slightly in time, mollification produces admissible smooth test functions. For fixed $\delta$ and $R$, these functions and the derivatives entering the weak formulation converge uniformly on their common compact support.

We first pass to the mollification limit and then let $\delta\downarrow0$. The continuity of $s\mapsto\int u_R(s,x)\,\d\mu_s(x)$ recovers the terminal boundary term from the time cutoff, yielding
\[
\begin{aligned}
&\int_{\bbr^\dm}u_R(t,x)\,\d\mu_t(x)
-\int_{\bbr^\dm}u_R(0,x)\,\d\rho_0(x)
\\
&\qquad
=\int_0^t\int_{\bbr^\dm}
\left[
\partial_su_R
-A_s(x-m_0)\cdot\nabla u_R
+\gamma\Delta u_R
\right]
\,\d\mu_s(x)\,\d s.
\end{aligned}
\]
Using the backward equation, the integrand on the right reduces to
\[
-uA_s(x-m_0)\cdot\nabla\eta_R
+\gamma u\Delta\eta_R
+2\gamma\nabla u\cdot\nabla\eta_R.
\]
All these terms are supported in $\{R\leq|x|\leq2R\}$. On this annulus,
\[
|\nabla\eta_R|\leq\frac{C}{R},
\qquad
|\Delta\eta_R|\leq\frac{C}{R^2},
\qquad
|A_s(x-m_0)|\leq C_T(1+|x|).
\]
The boundedness of $u$ and $\nabla u$, together with
\[
\sup_{0\leq s\leq T}\int_{\bbr^\dm}|x|^2\,\d\mu_s(x)<\infty,
\]
shows that every cutoff error tends to zero as $R\to\infty$. Passing to the
limit gives
\[
\int_{\bbr^\dm}u(t,x)\,\d\mu_t(x)
=\int_{\bbr^\dm}u(0,x)\,\d\rho_0(x).
\]
Since $u(t,x)=\psi(x)$, the desired identity follows.
\end{proof}

\section{Gaussian smoothing estimates}
\label{app:gaussian-smoothing}

We prove Lemma~\ref{lem:gaussian-smoothing} using a Gaussian
convolution inequality for centered probability measures. We first
state the result of Paulin and Whalley
\cite[Theorem~2.1]{paulin2026gaussianconvolution} with its assumptions
and constants, before applying it to the Gaussian component of our
solution. We then prove Lemma~\ref{lem:third-order-smoothing}
by a direct argument based on a moment identity and Talagrand's
transportation inequality.

\begin{theorem}[Paulin--Whalley, Theorem~2.1 of \cite{paulin2026gaussianconvolution}]
\label{thm:paulin-whalley-convolution}
Let $d\geq1$ and $p\geq1$. Suppose that $\mu$ is a probability
measure on $\bbr^d$ satisfying
\[
\int_{\bbr^d}x\,\d\mu(x)=0,
\qquad
\int_{\bbr^d}|x|^{2p}\,\d\mu(x)<\infty.
\]
Then, with $\mathsf g:=\mathcal N(0,I_d)$,
\[
W_p(\mu*\mathsf g,\mathsf g)
\leq
\frac{K_p}{1-(1-2^{-2p})^{1/p}}
\left(\int_{\bbr^d}|x|^{2p}\,\d\mu(x)\right)^{1/p},
\]
where
\[
C_p:=\bigl(\mathbb E|\xi|^p\bigr)^{1/p},
\qquad
K_p:=\max\left\{1,\frac{C_p}{2}+\frac13\right\},
\qquad \xi\sim\mathcal N(0,1).
\]
\end{theorem}

We use only the case $p=2$. Since $C_2=K_2=1$, the preceding
theorem gives the following explicit estimate for any centered
random vector $U$ with finite fourth moment and an independent
standard Gaussian vector $G$:
\begin{equation}
\label{eq:gaussian-convolution-literature}
W_2\bigl(\operatorname{Law}(U+G),\mathcal N(0,I_d)\bigr)
\leq
\frac{\bigl(\mathbb E|U|^4\bigr)^{1/2}}
{1-\sqrt{15/16}}.
\end{equation}
The centering assumption is essential here: without it, the
transported perturbation may contribute at first order in its size.

\begin{proof}[Proof of Lemma~\ref{lem:gaussian-smoothing}]
Recall that $Z\sim\mathcal N(0,Q)$ is independent of $X$.
The assertion is immediate
when $\varepsilon:=\|R\|_{\mathrm{op}}=0$, so assume $0<\varepsilon\leq1$.

We first suppose that $\mathbb E|X|^4<\infty$. Since $Q$ is
positive definite, we may define
\[
U:=Q^{-1/2}RX,\qquad G:=Q^{-1/2}Z.
\]
Then $U$ is centered and independent of $G$, and
$G$ is a standard Gaussian vector since
\[
\operatorname{Cov}(G)=Q^{-1/2}QQ^{-1/2}=I_d.
\]
The operator norm bounds then imply
\[
\mathbb E|U|^4
\leq\frac{\varepsilon^4}{q_-^2}\mathbb E|X|^4,
\]
and applying \eqref{eq:gaussian-convolution-literature} and using the
$\sqrt{q_+}$-Lipschitz map $x\mapsto Q^{1/2}x$ gives
\begin{equation}
\label{eq:gaussian-centered-perturbation}
W_2\bigl(\operatorname{Law}(RX+Z),\mathcal N(0,Q)\bigr)
\leq C\varepsilon^2\bigl(\mathbb E|X|^4\bigr)^{1/2},
\end{equation}
where $C$ depends only on $q_-$ and $q_+$.

To compare the Gaussian covariances, put
$S=R\Sigma_XR^\top$ and $T=Q^{-1/2}SQ^{-1/2}$.
Using the same standard Gaussian vector $G$, define
\[
Y:=Q^{1/2}(I_d+T)^{1/2}G.
\]
Since $Z=Q^{1/2}G$ and
\[
\operatorname{Cov}(Y)
=Q^{1/2}(I_d+T)Q^{1/2}=Q+S,
\]
the pair $(Z,Y)$ is a coupling of
$\mathcal N(0,Q)$ and $\mathcal N(0,Q+S)$.
Since $\mathbb E|AG|^2=\|A\|_F^2$ for any matrix $A$,
this coupling yields
\[
\begin{aligned}
&W_2\bigl(\mathcal N(0,Q),\mathcal N(0,Q+S)\bigr)\\
&\quad\leq
\bigl\|Q^{1/2}\bigl((I_d+T)^{1/2}-I_d\bigr)\bigr\|_F\\
&\quad\leq
\sqrt{q_+}\bigl\|(I_d+T)^{1/2}-I_d\bigr\|_F\\
&\quad\leq\frac{\sqrt{q_+}}{2}\|T\|_F
\leq\frac{\sqrt{q_+}}{2q_-}\|S\|_F.
\end{aligned}
\]
Here the third inequality follows from $T\geq0$ and
$\sqrt{1+s}-1\leq s/2$ for $s\geq0$. Since $S=R\Sigma_XR^\top$ is positive semidefinite, we have
\[
\|S\|_F\leq\operatorname{tr}S
=\mathbb E|RX|^2
\leq\varepsilon^2\mathbb E|X|^2,
\]
and the preceding Gaussian comparison therefore gives
\begin{equation}
\label{eq:gaussian-covariance-perturbation}
W_2\bigl(\mathcal N(0,Q),
\mathcal N(0,Q+R\Sigma_XR^\top)\bigr)
\leq C\varepsilon^2\mathbb E|X|^2.
\end{equation}
Thus the change in Gaussian covariance contributes an error
of the same order as the smoothed perturbation estimated in
\eqref{eq:gaussian-centered-perturbation}. Combining these
two estimates with the triangle inequality, we obtain
\[
\begin{aligned}
&W_2\bigl(\operatorname{Law}(RX+Z),
\mathcal N(0,Q+R\Sigma_XR^\top)\bigr)\\
&\qquad\leq C\varepsilon^2
\left((\mathbb E|X|^4)^{1/2}+\mathbb E|X|^2\right)\\
&\qquad\leq C\varepsilon^2(\mathbb E|X|^4)^{1/2},
\end{aligned}
\]
which establishes \eqref{eq:gaussian-smoothing-W2} with a
constant independent of $R$ and $Q$.

We now assume only that $\mathbb E|X|^2<\infty$.
To apply the preceding fourth-moment estimate, we approximate
$X$ by bounded centered random vectors. For $M>0$, define
\[
Y_M:=X\mathbf1_{\{|X|\leq M\}}
-\mathbb E[X\mathbf1_{\{|X|\leq M\}}],
\qquad
\delta_M:=\bigl(\mathbb E|X-Y_M|^2\bigr)^{1/2}.
\]
Since $X$ is centered, the difference $X-Y_M$ is the centered
tail of $X$, namely,
$X\mathbf1_{\{|X|>M\}}
-\mathbb E[X\mathbf1_{\{|X|>M\}}]$.
Using the identity for the second moment of a centered random
vector, we obtain
\[
\begin{aligned}
\delta_M^2
&=\mathbb E\bigl[|X|^2\mathbf1_{\{|X|>M\}}\bigr]
 -\bigl|\mathbb E[X\mathbf1_{\{|X|>M\}}]\bigr|^2\\
&\leq\mathbb E\bigl[|X|^2\mathbf1_{\{|X|>M\}}\bigr]
\longrightarrow0
\qquad\text{as }M\to\infty,
\end{aligned}
\]
where the convergence follows from dominated convergence, since
$X$ has a finite second moment.

For each fixed $M$, the random vector $Y_M$ is bounded and
centered, and hence has a finite fourth moment. We may
therefore apply \eqref{eq:gaussian-centered-perturbation}
to $Y_M$. On the other hand, coupling $RX+Z$ and $RY_M+Z$
using the same random vectors $X$ and $Z$ gives
\[
W_2\bigl(\operatorname{Law}(RX+Z),
\operatorname{Law}(RY_M+Z)\bigr)
\leq\varepsilon\delta_M.
\]
Combining these estimates by the triangle inequality yields
\[
W_2\bigl(\operatorname{Law}(RX+Z),\mathcal N(0,Q)\bigr)
\leq\varepsilon\delta_M+C_M\varepsilon^2,
\]
where $C_M$ is independent of $R$ and $Q$.

The covariance comparison
\eqref{eq:gaussian-covariance-perturbation}, applied to
the original covariance $\Sigma_X$, now gives
\[
W_2\left(\operatorname{Law}(RX+Z),
\mathcal N(0,Q+R\Sigma_XR^\top)\right)
\leq\varepsilon\delta_M+C'_M\varepsilon^2.
\]
For each fixed $M$, dividing by $\varepsilon$ and letting
$\varepsilon\to0$ eliminates the term $C'_M\varepsilon$,
leaving an upper bound $\delta_M$ for the limit superior.
Since this bound holds uniformly in $R$ and $Q$ and
$\delta_M\to0$ as $M\to\infty$, the desired
$o(\varepsilon)$ estimate follows.

\end{proof}

We now turn to the third-order estimate for bounded perturbations. Its
proof does not rely on Theorem~\ref{thm:paulin-whalley-convolution}. It
uses instead the following inequality of Talagrand
\cite{talagrand1996transportation}: for every probability measure
$\mu$ on $\bbr^d$,
\begin{equation}
\label{eq:talagrand}
W_2^2(\mu,\mathsf g)\leq2\mathcal H(\mu\mid\mathsf g),
\qquad
\mathsf g=\mathcal N(0,I_d),
\end{equation}
where $\mathcal H$ denotes the relative entropy defined in
Proposition~\ref{prop:Gaussian-reduction-minimizer}, with the
convention $\mathcal H(\mu\mid\mathsf g)=+\infty$ if $\mu$ is not
absolutely continuous with respect to $\mathsf g$.

\begin{proof}[Proof of Lemma~\ref{lem:third-order-smoothing}]
We use the normalization from the proof of
Lemma~\ref{lem:gaussian-smoothing} and define
\[
U:=Q^{-1/2}RX,
\qquad
G:=Q^{-1/2}Z,
\qquad
T:=\mathbb E[UU^\top]=Q^{-1/2}R\Sigma_XR^\top Q^{-1/2}.
\]
Then $G$ is a standard Gaussian vector independent of $U$, and
\[
|U|\leq\eta:=\frac{\varepsilon M}{\sqrt{q_-}}
\quad\text{almost surely},
\qquad
0\leq T\leq\eta^2I_d.
\]
Let $\nu:=\operatorname{Law}(U+G)$ and
$\overline\nu:=\mathcal N(0,I_d+T)$. Since $RX+Z=Q^{1/2}(U+G)$ and
$Q^{1/2}(I_d+T)Q^{1/2}=Q+R\Sigma_XR^\top$, the
$\sqrt{q_+}$-Lipschitz map $x\mapsto Q^{1/2}x$ gives
\begin{equation}
\label{eq:third-order-normalization}
W_2\left(
\operatorname{Law}(RX+Z),
\mathcal N(0,Q+R\Sigma_XR^\top)
\right)
\leq\sqrt{q_+}\,W_2(\nu,\overline\nu).
\end{equation}
It therefore suffices to estimate $W_2(\nu,\overline\nu)$. Throughout
the proof, we assume $\eta\leq1/2$, which holds for
$\varepsilon\leq\varepsilon_0:=\sqrt{q_-}/(2M)$, and we denote by $C$
a constant depending only on $d$.

We first derive an identity for a weighted $L^2$ distance between the
densities of $\nu$ and $\overline\nu$. Let $\varphi$ be the density of
$\mathsf g=\mathcal N(0,I_d)$, and let $V\sim\mathcal N(0,T)$ be
independent of $G$, so that $\overline\nu=\operatorname{Law}(V+G)$.
For a random vector $W$ independent of $G$, the density of $W+G$ is
\[
x\longmapsto\varphi(x)\,
\mathbb E\left[\exp\left(x\cdot W-\tfrac12|W|^2\right)\right].
\]
Moreover, for all $u,u'\in\bbr^d$, the identity
$x\cdot(u+u')-\tfrac12|u|^2-\tfrac12|u'|^2-\tfrac12|x|^2
=u\cdot u'-\tfrac12|x-u-u'|^2$ gives
\[
\int_{\bbr^d}
\exp\left(x\cdot u-\tfrac12|u|^2\right)
\exp\left(x\cdot u'-\tfrac12|u'|^2\right)
\varphi(x)\,\d x
=e^{u\cdot u'}\int_{\bbr^d}\varphi(x-u-u')\,\d x
=e^{u\cdot u'}.
\]
Let $f_\nu$ and $f_{\overline\nu}$ denote the densities of $\nu$ and
$\overline\nu$, and let $U'$ and $V'$ be copies of $U$ and $V$ such that
$U,U',V,V'$ are independent. Applying the preceding identity together
with Tonelli's theorem to each of the three products, we obtain
\begin{equation}
\label{eq:weighted-L2-exponential}
\int_{\bbr^d}
\frac{(f_\nu-f_{\overline\nu})^2}{\varphi}\,\d x
=\mathbb E\,e^{U\cdot U'}
-2\,\mathbb E\,e^{U\cdot V}
+\mathbb E\,e^{V\cdot V'}.
\end{equation}
The three expectations in \eqref{eq:weighted-L2-exponential}
are finite. Indeed, $|U\cdot U'|\leq\eta^2$,
while conditioning on $U$ and on $V$, respectively, gives
\[
\mathbb E\,e^{U\cdot V}
=\mathbb E\,e^{\frac12U^\top TU}
\leq e^{\eta^4/2},
\qquad
\mathbb E\,e^{V\cdot V'}
=\mathbb E\,e^{\frac12V^\top TV}
=\det(I_d-T^2)^{-1/2},
\]
and the last determinant is finite because
$\|T\|_{\mathrm{op}}\leq\eta^2<1$.

For $k\geq0$ and a random vector $W$ with finite moments of all
orders, we write $m_k(W):=\mathbb E[W^{\otimes k}]$ for its $k$-th
moment tensor, and we denote by $\langle\cdot,\cdot\rangle$ and
$|\cdot|$ the Euclidean inner product and norm on
$(\bbr^d)^{\otimes k}$. If $W$ and $W'$ are independent, then
\[
\mathbb E\bigl[(W\cdot W')^k\bigr]
=\bigl\langle m_k(W),m_k(W')\bigr\rangle.
\]
We now expand the exponentials in
\eqref{eq:weighted-L2-exponential} and integrate termwise, which is
justified since
$\mathbb E\,e^{|U\cdot V|}$ and $\mathbb E\,e^{|V\cdot V'|}$ are finite
by the bounds above and the symmetry of the law of $V$. Then we
obtain
\begin{equation}
\label{eq:hermite-identity}
\int_{\bbr^d}
\frac{(f_\nu-f_{\overline\nu})^2}{\varphi}\,\d x
=\sum_{k=0}^\infty\frac{1}{k!}
\bigl|m_k(U)-m_k(V)\bigr|^2.
\end{equation}
The terms with $k\leq2$ vanish, since $U$ and $V$ are centered and
have the same covariance $T$.

We next estimate the remaining terms. The bound $|U|\leq\eta$ gives
$|m_k(U)|^2=\mathbb E[(U\cdot U')^k]\leq\eta^{2k}$, and hence
\[
\sum_{k=3}^\infty\frac{1}{k!}|m_k(U)|^2
\leq\eta^6e^{\eta^2}\leq2\eta^6.
\]
For the Gaussian vector $V$, we have $|m_0(V)|^2=1$, $m_1(V)=0$, and
$|m_2(V)|^2=\|T\|_F^2$, so that
\[
\sum_{k=3}^\infty\frac{1}{k!}|m_k(V)|^2
=\mathbb E\,e^{V\cdot V'}-1-\frac12\|T\|_F^2
=\det(I_d-T^2)^{-1/2}-1-\frac12\|T\|_F^2.
\]
Let $\tau_1,\ldots,\tau_d\in[0,\eta^2]$ be the eigenvalues of $T$, and
set $\sigma:=\frac12\sum_{i=1}^d(\tau_i^2+\tau_i^4)$. The elementary
inequalities $-\log(1-y)\leq y+y^2$ for $0\leq y\leq1/2$ and
$e^\sigma\leq1+\sigma+\frac12\sigma^2e^\sigma$ for $\sigma\geq0$,
together with $\sigma\leq d\eta^4\leq d$, yield
\[
\begin{aligned}
\det(I_d-T^2)^{-1/2}
&=\exp\left(-\frac12\sum_{i=1}^d\log(1-\tau_i^2)\right)
\leq e^\sigma\\
&\leq1+\frac12\|T\|_F^2
+\frac12\sum_{i=1}^d\tau_i^4
+\frac12\sigma^2e^\sigma
\leq1+\frac12\|T\|_F^2+C\eta^8.
\end{aligned}
\]
Since $|m_k(U)-m_k(V)|^2\leq2|m_k(U)|^2+2|m_k(V)|^2$, the identity
\eqref{eq:hermite-identity} and the preceding bounds give
\begin{equation}
\label{eq:weighted-L2-third-order}
\int_{\bbr^d}
\frac{(f_\nu-f_{\overline\nu})^2}{\varphi}\,\d x
\leq4\eta^6+C\eta^8\leq C\eta^6.
\end{equation}

It remains to pass from this weighted $L^2$ bound to the Wasserstein
distance. Since $(I_d+T)^{-1}\leq I_d$, the Gaussian density
$f_{\overline\nu}$ satisfies
$f_{\overline\nu}\geq\det(I_d+T)^{-1/2}\varphi$. Using Jensen's
inequality for the logarithm, we therefore obtain
\[
\begin{aligned}
\mathcal H(\nu\mid\overline\nu)
&=\int_{\bbr^d}
\log\left(\frac{f_\nu}{f_{\overline\nu}}\right)\,\d\nu
\leq\log\left(
\int_{\bbr^d}\frac{f_\nu^2}{f_{\overline\nu}}\,\d x
\right)\\
&=\log\left(
1+\int_{\bbr^d}
\frac{(f_\nu-f_{\overline\nu})^2}{f_{\overline\nu}}\,\d x
\right)
\leq\det(I_d+T)^{1/2}
\int_{\bbr^d}
\frac{(f_\nu-f_{\overline\nu})^2}{\varphi}\,\d x.
\end{aligned}
\]
Together with \eqref{eq:weighted-L2-third-order} and
$\det(I_d+T)\leq(1+\eta^2)^d$, this gives
$\mathcal H(\nu\mid\overline\nu)\leq C\eta^6$. Finally, the linear map
$x\mapsto(I_d+T)^{-1/2}x$ pushes $\overline\nu$ forward to
$\mathsf g$ and preserves the relative entropy, while its inverse is
$(1+\eta^2)^{1/2}$-Lipschitz. Denoting by $\nu'$ the image of $\nu$
under this map and applying \eqref{eq:talagrand}, we conclude that
\[
W_2^2(\nu,\overline\nu)
\leq(1+\eta^2)W_2^2(\nu',\mathsf g)
\leq2(1+\eta^2)\mathcal H(\nu'\mid\mathsf g)
=2(1+\eta^2)\mathcal H(\nu\mid\overline\nu)
\leq C\eta^6.
\]
Combining this bound with \eqref{eq:third-order-normalization} and the
definition of $\eta$ proves \eqref{eq:third-order-smoothing-W2}, with a
constant depending only on $d$, $M$, $q_-$, and $q_+$.

Suppose finally that all third moments of $X$ vanish. Then
$m_3(U)=0$, because $U$ is a linear image of $X$, and $m_3(V)=0$,
because $V$ is a centered Gaussian vector. The sum in
\eqref{eq:hermite-identity} therefore starts at $k=4$, and the same
argument gives
\[
\int_{\bbr^d}
\frac{(f_\nu-f_{\overline\nu})^2}{\varphi}\,\d x
\leq2\sum_{k=4}^\infty\frac{\eta^{2k}}{k!}+C\eta^8
\leq C\eta^8.
\]
Repeating the last step with this bound yields
$W_2(\nu,\overline\nu)\leq C\eta^4$, which proves the second
assertion.
\end{proof}

\bibliographystyle{amsplain}
\bibliography{lit_verified_final}

\end{document}